\documentclass[a4paper,10pt, english, reqno]{amsart}
\usepackage{tikz-cd}
\usetikzlibrary{cd}
\usepackage{url}

\title[$p$-adic families for irregular forms]{Arithmetic of $p$-irregular modular forms: families and $p$-adic $L$-functions}
\author{Adel Betina}

\address{Faculty of Mathematics, University of Vienna, Oskar-Morgenstern-Platz 1, A-1090 Wien, Austria}

\email{adelbetina@gmail.com }

\author{Chris Williams}
\address{Mathematics Institute, Zeeman Building, University of Warwick, CV4 7AL, UK}
\email{Christopher.D.Williams@warwick.ac.uk}
\date{\today}

\usepackage[title]{appendix}

\newcommand{\sar}[2]{\ar@{}[#1]|-*[@]{#2}}
\newcommand{\BCf}{f_{/F}}

\usepackage[new]{old-arrows}
\numberwithin{equation}{section}

\usepackage{titlesec}
\titleformat{\subsubsection}[runin]
{\normalfont\itshape}
{\thesubsubsection.}
{0.5em}
{}
[.\hspace*{6pt}]

\titleformat{\subsection}[runin]
{\normalfont\bfseries}
{\thesubsection.}
{0.5em}
{}
[.\hspace*{6pt}]

\titleformat{\section}
{\large\bfseries\centering}
{\thesection.}
{0.5em}
{}
[]

\titlespacing*{\section}{0pt}{1.2\baselineskip}{0.2\baselineskip}
\titlespacing*{\subsection}{0pt}{0.8\baselineskip}{\baselineskip}
\titlespacing*{\subsubsection}{0pt}{0.7\baselineskip}{\baselineskip}

\makeatletter 
\def\input@path{{../}} 
\makeatother
\usepackage{preamble}

\newcommand{\Y}{Y_N}
\newcommand{\OS}{\Lambda}
\newcommand{\OSl}{\Lambda_\lambda}

\newcommand{\eps}{\varepsilon}
\newcommand{\cMbar}{\overline{\cM}}
\newcommand{\Gp}{\mathrm{Gal}_p}

\newcommand{\cL}{\mathcal{L}}

\newcommand{\nbc}{\mathrm{nbc}}

\usepackage{MnSymbol}

\begin{document}

%
%

\begin{abstract}
Let $f_{\mathrm{new}}$ be a classical newform of weight $\geq 2$ and prime to $p$ level. We study the arithmetic of $f_{\mathrm{new}}$ and its unique $p$-stabilisation $f$ when $f_{\mathrm{new}}$ is $p$-irregular, that is, when its Hecke polynomial at $p$ admits a single repeated root. In particular, we study $p$-adic weight families through $f$ and its base-change to an imaginary quadratic field $F$ where $p$ splits, and prove that the respective eigencurves are both Gorenstein at $f$. We use this to construct a two-variable $p$-adic $L$-function over a Coleman family through $f$, and a three-variable $p$-adic $L$-function over the base-change of this family to $F$. We relate the two- and three-variable $p$-adic $L$-functions via $p$-adic Artin formalism. These results are used in work of Xin Wan to prove the Iwasawa Main Conjecture in this case. 

In an appendix, we prove results towards Hida duality for modular symbols, constructing a pairing between Hecke algebras and families of overconvergent modular symbols and proving that it is non-degenerate locally around any cusp form. This allows us to control the sizes of (classical and Bianchi) Hecke algebras in families. 
\end{abstract}

\maketitle

\section{Introduction}
Let $f_{\mathrm{new}} \in S_k(\Gamma_0(M))$ be a classical cuspidal newform with $k\geq 2$ and $p\nmid M$. The Bloch--Kato conjecture predicts that the analytic $L$-function attached to $f_{\mathrm{new}}$ encodes deep arithmetic properties of $f_{\mathrm{new}}$. One of the main tools we have for proving such links is \emph{Iwasawa theory}, which seeks to recast and prove Bloch--Kato in $p$-adic language. In particular, let $f$ be a $p$-stabilisation of $f_{\mathrm{new}}$ to level $N = Mp$; then the \emph{Iwasawa Main Conjecture (IMC)} for $f$ describes the arithmetic of its $\Lambda$-adic Selmer group in terms of its \emph{$p$-adic $L$-function}. The IMC is known to hold under a number of assumptions, including the conjecture that $f$ is \emph{$p$-regular}, that is, that the roots $\alpha_p,\beta_p$ of the Hecke polynomial
\begin{equation}\label{eq:hecke poly}
X^2 - a_p(f_{\mathrm{new}})X + p^{k-1}
\end{equation}
are distinct, where $a_p(f_{\mathrm{new}})$ is the $T_p$-eigenvalue of $f_{\mathrm{new}}$ \cite{SU14, Wan14}. Such results have important applications to the Bloch--Kato conjecture for $f$ (see, for example, \cite{JSW15} for this when $f_{\mathrm{new}}$ is attached to a suitable elliptic curve). 

For modular forms $f$ where the IMC is known, a key tool in the proof is the existence of multi-variable $p$-adic $L$-functions, which interpolate the $p$-adic $L$-functions of classical modular forms as they vary in $p$-adic families through $f$. In the $p$-irregular case, such functions had not previously been constructed. In this paper, we construct:
\begin{itemize}\setlength{\itemsep}{0pt}
	\item a 2-variable $p$-adic $L$-function over a Coleman family through a $p$-irregular form $f$,
	\item and a 3-variable $p$-adic $L$-function over the base-change of such a Coleman family to an imaginary quadratic field where $p$ splits.
\end{itemize}
We also relate these multi-variable $p$-adic $L$-functions, proving they satisfy $p$-adic Artin formalism. Fundamentally using the results of this paper, in \cite{Wan20} Wan has proved the IMC for classical modular forms of level prime to $p$ without requiring a regularity assumption.

There is a well-established conjecture that $p$-irregular forms of weight $k\geq 2$ never exist. This appears, however, completely out of reach at present. Our motivation in studying this case is three-fold: firstly we facilitate an unconditional proof of the IMC; secondly we develop methods that should work in more general settings, such as Hilbert modular forms, where $p$-irregular forms are known to exist (e.g.\ \cite{Chi15}); and finally, in line with recent work of Molina Blanco, we add to what is known about $p$-irregular forms in the hope that this may lead to a contradiction and a proof of the conjecture. We comment more on this in \S\ref{sec:motivation} below.

The $p$-irregular case has been omitted from previous treatments of this topic because the eigenvariety is harder to control in this setting. Accordingly, during the course of our study, we develop new methods for studying eigenvarieties that should be of use much more generally. In particular, at one point we need to control the size of Hecke algebras in families. For classical modular forms, it is possible to do this using overconvergent modular forms and Hida duality; but over imaginary quadratic fields -- the \emph{Bianchi} case -- these tools are not available. In an appendix, we prove results towards an analogue of Hida duality for modular symbols; in particular, we use evaluation maps to construct a pairing between Bianchi modular symbols and the Hecke algebra, and prove that it is non-degenerate locally around any cusp form. We believe these results to be of independent interest, since they require no non-criticality assumption and seem likely to generalise well to overconvergent cohomology in other settings.

\subsection{Families of $p$-adic $L$-functions through $p$-irregular forms}

We state our main results. Let $f_{\mathrm{new}}$ be a $p$-irregular form, let $\alpha_p$ be the unique (repeated) root of the Hecke polynomial \eqref{eq:hecke poly} at $p$, and let $\mathcal{F}$ be a Coleman family through the (unique) $p$-stabilisation 
\[
	f \defeq f_{\mathrm{new}}(z) - \alpha_p f_{\mathrm{new}}(pz).
\] 
This family is captured geometrically via the Coleman--Mazur eigencurve $\cC$: in particular, we can consider $f$ as a point $x_f \in \cC$, and $\f$ as a neighbourhood $V$ of $x_f$. Attached to any classical point $y \in V$, corresponding to a modular form $g$ of weight $k_g+2$, we have $p$-adic $L$-functions $L_p^\pm(g)$, which interpolate the critical values $L(g,\chi,j+1)$ for all $0 \leq j \leq k_g$ and Dirichlet characters $\chi$ of $p$-power conductor with $\chi(-1)(-1)^j = \pm1$. These $p$-adic $L$-functions are each supported on different halves of weight space, and in particular, the $p$-adic $L$-function $L_p(g) \defeq L_p^+(g) + L_p^-(g)$ is supported on all of weight space and interpolates all the critical values. We show that:

\begin{theorem-intro}\label{intro:interpolation}	
	Up to shrinking $V$, there exist unique two-variable $p$-adic $L$-functions $L_p^\pm(V)$ over $V$, such that at any classical point $y \in V$ corresponding to a modular form $g$, we have
\[
	L_p^\pm(V,y) = c_g^\pm L_p^\pm(g),
\]
where the $c_g^\pm \in \overline{\Q}_p$ are non-zero $p$-adic periods depending only on $g$.
\end{theorem-intro}

Taking the Amice transform, $L_p(V) \defeq L_p^+(V) + L_p^-(V)$ can be viewed as a rigid-analytic function 
	\[
		\mathcal{L}_p : V \times \mathscr{X}(\Zp^\times) \longrightarrow \Cp,
	\]
where $\mathscr{X}(\Zp^\times)$ is the space of $p$-adic rigid characters on $\Zp^\times$, and $\mathcal{L}_p$ satisfies the following interpolation property; for:
\begin{itemize} 
\item any such $y \in V$ and $g$ as above, with $g$ of weight $k_g+2$,
\item any Dirichlet character $\chi$ of conductor $p^r > 1$,
\item and any integer $0 \leq j \leq k_g$,
\end{itemize}
we have
	\begin{equation}\label{eqn:interpolation intro}
		\mathcal{L}_p(g,\chi(z)z^j) = c_g^\pm \frac{p^{r(j+1)}}{\alpha_g^r\tau(\chi^{-1})}\cdot\frac{\Lambda(g,\chi^{-1},j+1)}{\Omega_g^\pm},
	\end{equation}
	where $\chi(-1)(-1)^j = \pm 1,$ $\alpha_g$ is the $U_p$-eigenvalue of $g$, $\tau(\chi^{-1})$ is a Gauss sum, $\Lambda(g,\chi^{-1},j+1)$ is the normalised $L$-function of $g$, and $\Omega^\pm_g$ are the complex periods of $g$. (When $\mathrm{cond}(\chi)=1$, there is also an exceptional factor to consider, studied comprehensively in \cite{MTT86}).

We also have a version of Theorem \ref{intro:interpolation} for the base-change of $\f$ to an imaginary quadratic field $F$ in which $p$ splits. In this case, there are no signs to consider, and the $p$-adic $L$-function of a single form is naturally a locally analytic distribution on the Galois group $\Gp$ of the maximal abelian extension of $F$ unramified outside $p\infty$, which is a two-dimensional $p$-adic group. Working with families in this setting is much harder, due to the presence of cuspidal Hecke eigensystems  in degree two of the cohomology of Bianchi hyperbolic threefolds. Accordingly, in this case we make use of the technical machinery developed in \cite{BW18}, in which these families were carefully studied. We prove:

\begin{theorem-intro}\label{intro:interpolation bianchi}	Up to shrinking $V$, there exists a unique three-variable $p$-adic $L$-function $L_p(V_{/F})$ over $V,$ such that at any classical point $y \in V$ corresponding to a cuspform $g$, we have
\[
	L_p(V_{/F},y) = c_g' L_p(g_{/F}),
\]
where $g_{/F}$ is the base-change of $g$ to $F$, $L_p(g_{/F})$ is the $p$-adic $L$-function of $g_{/F}$ and $c_g'$ is a non-zero $p$-adic period depending only on $g$.
\end{theorem-intro}

These base-change forms are \emph{Bianchi modular forms}. Again the three-variable $p$-adic $L$-function admits a precise interpolation formula as a rigid analytic function 
\[
	\mathcal{L}_p : V \times \mathscr{X}(\Gp) \longrightarrow \Cp,
\] 
which takes exactly the form of \cite[Thm.~A]{BW18}. In particular, for each classical $g\in V$ and every Hecke character $\varphi$ critical for $g$ of $p$-power conductor, we have
\[
	\mathcal{L}_p(g, \varphi_{p-\mathrm{fin}}) = c_g'\cdot A(g_{/F},\varphi) \cdot \frac{\Lambda(g_{/F},\ \varphi)}{\Omega_{g_{/F}}}, 
\]
where $A(g_{/F},\varphi)$ is a precise interpolation factor similar to \eqref{eqn:interpolation intro} above and described explicitly in \cite[Thm.~7.4]{Wil17}. We may take the complex period $\Omega_{g_{/F}}$ to be an explicit algebraic multiple of $\Omega_{g}^+\Omega_g^-$, described in \S\ref{sec:artin formalism}.

We prove both of these theorems using the same approach -- namely, that of overconvergent cohomology -- and hence treat them at the same time by working over a field $K$ that we take to be either $\Q$ or imaginary quadratic with $p$ split. Overconvergent cohomology was introduced in \cite{Ste94}, and then used to construct the (one-variable, cyclotomic) $p$-adic $L$-function attached to a classical modular form in \cite{PS11,PS12,Bel12} and the (two-variable) $p$-adic $L$-function attached to a Bianchi modular form in \cite{Wil17}. In all of these papers, the $p$-adic $L$-function of a form $f$ was constructed as the Mellin transform of a certain (canonical up to scalar) class $\Phi_f$ in the relevant overconvergent cohomology group. For a detailed survey of this approach, and a diagram illustrating this method of construction, see \cite[\S1]{BW20}.

We consider variation of the overconvergent cohomology (in degree 1) over the Coleman--Mazur and Bianchi eigencurves, as studied in \cite{Bel12} (classical case) and \cite{BW18} (Bianchi case). In each case, we show that the local ring of the eigencurve is Gorenstein at the $p$-irregular form $f$, and use this to deduce the existence of a family of overconvergent eigenclasses in the cohomology, interpolating the classes $\Phi_g$ as $g$ varies in a Coleman family. The Mellin transform of this family is the desired $p$-adic $L$-function, which now has two or three variables when $K$ is $\Q$ or imaginary quadratic respectively.

We comment briefly on why new arguments are required in the $p$-irregular case. The previous most general constructions of multi-variable $p$-adic $L$-functions, in \cite{Bel12} and \cite{BW18}, treat two cases separately, using the arithmetic of $f$ to deduce results about the structure of the eigenvariety.
\begin{itemize}
	\item Suppose $f$ is \emph{non-critical}, that is, the classical and overconvergent generalised eigenspaces are isomorphic. If $f$ is $p$-regular, then both are thus one-dimensional, from which the eigencurve can be shown to be \'etale over weight space at $f$.
	
	\item Suppose $f$ is \emph{critical}, so that the overconvergent generalised eigenspace is strictly bigger than the classical one. Then the eigencurve is smooth at $f$.
\end{itemize}

In both cases, the structure of the eigencurve can be used to deduce that the coherent sheaf of overconvergent cohomology groups over the eigencurve is in fact a line bundle at $f$ (that is, it is locally free of rank one over the Hecke algebra at $f$), and the (actual) eigenspace of $f$ in the overconvergent cohomology is one-dimensional, and these facts combine to give the construction. 

If $f$ is $p$-irregular, then it is non-critical, but the classical generalised eigenspace is no longer one-dimensional. The eigencurve is not \'etale over weight space, and it is not clear whether it is smooth at $f$. We instead deduce that the overconvergent cohomology is a line bundle at $f$ via a careful study of the classical and overconvergent Hecke algebras, showing that the local ring of the eigencurve -- which is a localised Hecke algebra -- is Gorenstein at $f$. In the Bianchi situation, this is quite subtle, since non-vanishing of $\hc{2}$ can provide an obstruction to weight specialisation. To prove Gorensteinness in this case, in \S\ref{sec:hecke alg in families} we use deformation-theoretic arguments to study the specialisation. We then deduce that the $f$-generalised eigenspace in overconvergent cohomology (and its dual) is free of rank one over the local Hecke algebra by combining properties of the classical spaces, non-criticality and formal properties of Gorenstein rings.

There are plenty of examples of weight one classical eigenforms which are irregular at $p$. Such eigenforms have critical slope. The recent works \cite{BDP,BD} studied the geometry of the eigencurve at such points following a new approach, hence deducing some arithmetic properties on trivial zeros of their adoint $p$-adic $L$-functions (that is, the Kubota--Leopoldt and Katz $p$-adic $L$-functions).  In contrast to the results of this paper, the local ring at these irregular weight one eigenforms is never Gorenstein and their associated overconvergent generalised eigenspace contains non-classical $p$-adic modular forms; hence the construction of the two-variable $p$-adic $L$-function around these points remains an open and challenging question in Iwasawa theory.

\subsection{$p$-adic Artin formalism}
A third tool required for the proof of the IMC is a formula relating the $p$-adic $L$-functions of $f$ and its base-change $f_{/F}$, for $F$ the imaginary quadratic field above. In particular, let $L_p^{\mathrm{cyc}}(f_{/F})$ be the restriction of $L_p(f_{/F})$ to the cyclotomic line, a copy of $\Zp^\times$ inside $\Gp$. Then we have:

\begin{theorem-intro}\label{thm:artin formalism-main}
	$L_p^{\mathrm{cyc}}(f_{/F}) = L_p(f)L_p(f \otimes \chi_{F/\Q})$ as distributions on $\Zp^\times$.
\end{theorem-intro}

Here $L_p(f \otimes \chi_{F/\Q})$ is the $p$-adic $L$-function of $f$ twisted by the quadratic character $\chi_{F/\Q}$ associated to $F$. This interpolates the critical $L$-values $\Lambda(f,\chi_{F/\Q}\chi,j+1)$ for $\chi$ as above. To prove Theorem \ref{thm:artin formalism-main}, we use the strategy of \cite{BW18}; indeed, this factorisation holds at a Zariski-dense set of classical points in the eigencurve, as can be seen from the interpolation formula and growth properties of the respective $p$-adic $L$-functions, and this interpolates to give the factorisation at the level of two-variable $p$-adic $L$-functions. We prove the theorem by specialising to the form $f$. Theorem \ref{thm:artin formalism-main} \emph{cannot} be seen directly from the interpolation and growth properties at $f$, since $L_p(f)L_p(f \otimes \chi_{F/\Q})$ is $(k+1)$-admissible and there are precisely $k+1$ critical integers for $L(f,s)$.

\subsection{Motivation and Tate's conjecture}\label{sec:motivation}
Tate's conjecture on the dimension of Chow groups of smooth projective varieties over finite fields \cite[\S2]{Mil94} predicts the non-existence of $p$-irregular cuspforms of weight $\geq 2$. Still, a proof of this conjecture eludes mathematicians to this day, and it stands as one of the hardest open questions in arithmetic geometry.

One might hope to prove the non-existence of such $p$-irregular points \emph{without} appealing to Tate's conjecture, but instead by deriving a contradiction from Iwasawa theory.
To expand on this, recent work of Molina Blanco \cite{Mol20} studies $p$-adic variation of the $L$-function attached to a $p$-irregular form $f$. In addition to the standard $p$-adic $L$-function considered in the present paper, he proves the existence of an attached `extremal' $p$-adic $L$-function, interpolating the same ($p$-depleted) $L$-values with different interpolation factors at $p$. He shows that this $p$-adic $L$-function can be concretely linked to the two-variable $p$-adic $L$-function $\cL_p$ of this paper: precisely, it can be obtained by differentiating $\cL_p$ in the weight direction and then specialising at $f$. He further speculates that the existence of such an object might lead to a contradiction; and one could hope to derive such a contradiction by exploiting the Iwasawa theory of $f$, as explored here.

Beyond this, there are several more concrete reasons for studying this case. Most directly, Theorems \ref{intro:interpolation}, \ref{intro:interpolation bianchi} and \ref{thm:artin formalism-main} are used fundamentally in \cite{Wan20} to prove the IMC without any $p$-regularity assumption. Our study should also prove a test-case for more general `badly behaved' situations; for example, situations where cohomological $p$-irregular forms do exist (e.g.\ Hilbert modular forms), and other settings where classical constructions of families $p$-adic $L$-functions break down (where classical generalised eigenspaces are not 1-dimensional, or smoothness of the eigenvariety is not known).

\subsection*{Structure of the paper} In \S\ref{sec:classical cohomology}, we provide a study of classical cohomology groups and Hecke algebras localised at $p$-irregular forms. In \S\ref{sec:oc cohomology}, we recap overconvergent cohomology and the construction of $p$-adic $L$-functions for single forms. The heart of the paper is \S\ref{sec:geometry at irregular points}, where use overconvergent cohomology in families to study the geometry of the classical and Bianchi eigencurves at irregular points, and prove our main Gorensteinness result. In \S\ref{sec:multi variable} we use this to construct the multi-variable $p$-adic $L$-functions, and conclude in \S\ref{sec:artin formalism} by proving $p$-adic Artin formalism.

\subsection*{Acknowledgements} 
We are grateful to Xin Wan, whose questions motivated this project, and for extensive discussions on the subject. We also thank the referees for their valuable comments and corrections. A.B.\ was supported by the EPSRC Grant EP/R006563/1 and by START-Prize Y966 of the Austrian Science Fund (FWF). C.W.\ was supported by an EPSRC Postdoctoral Fellowship EP/T001615/1.

\section{Classical cohomology at irregular forms}\label{sec:classical cohomology}

\subsection{Basic notation}\label{sec:notation}
Let $p$ be a prime, and let $F$ be an imaginary quadratic field in which $p$ splits as $\pri\pribar$. Let $K$ be either $\Q$ or $F$, let $D = \mathrm{disc}(K)$, and let $\cO_K$ be its ring of integers. Let $\Sigma_{\R}$ denote the set of complex embeddings of $K$. Let $G = \mathrm{Res}_{K/\Q}\GLt$. Let $\cO_p \defeq \cO_K\otimes_{\Z}\Zp$. 

Let $M$ be an integer coprime to $p$ and $D$ and let and $N = Mp$. Let $f_{\mathrm{new}}$ be either:
\begin{itemize}
\item[(a)] a classical newform of weight $\lambda = k+2 \geq 2$ and level $\Gamma_0(M)$ (if $K=\Q$),
\item[(b)] or the base-change of such a form to $F$, a Bianchi modular eigenform of weight $\lambda = (k,k) \geq 2$ and level $\Gamma_0(M\cO_K)$ (if $K = F$).
\end{itemize}
In case (a) (resp.\ (b)) $\lambda$ corresponds to the character $\smallmatrd{s}{}{}{t} \mapsto s^k$ (resp.\ $\smallmatrd{s}{}{}{t} \mapsto s^k\overline{s}^k$) of the diagonal torus in $G$. In case (b), we assume the original form does not have CM by $F$, so $f_{\mathrm{new}}$ is cuspidal. We write $S_\lambda(\Gamma_0(N))$ for the space of cusp forms of weight $\lambda$ and level $\Gamma_0(N)$.

If $K=\Q$, let $\alpha_p$ be a root of the Hecke polynomial\footnote{Note the shift by 2 here; we use $k+1$ since our convention is that $f$ has weight $k+2$.} (at $p$) $X^2 - a_p(f_{\mathrm{new}})X + p^{k+1}$ of $f_{\mathrm{new}}$, and if $K\neq \Q$, let $\alpha_{\pri} = \alpha_{\pribar} = \alpha_p$ be the corresponding roots  at $\pri$ and $\pribar$. Let $f$ be the corresponding $p$-stabilisation of $f_{\mathrm{new}}$ to level $\Gamma_0(N)$ (that is, an eigenform with $U_{\pri}$-eigenvalue $\alpha_{\pri}$ for each $\pri|p$).

Let $\bH_{N}^{\mathrm{tame}}$ denote the abstract tame Hecke algebra at level $\Gamma_0(N)$ at $p$, that is, the free $\Z$-algebra generated by the Hecke operators $T_v$ (for finite places $v\nmid N$ of $K$); we work at level $\Gamma_0$, so exclude the diamond operators. Let
\[
	\bH_N \defeq \bH_N^{\mathrm{tame}}[ \{U_{\pri} : \pri|p\}].
\] 
 If $\cM$ is a module on which $\bH_{N}$ acts, we write $\cM[f]$ (resp.\ $\cM_{f}$) for the eigenspace (resp.\ generalised eigenspace) upon which $\bH_{N}$ acts with the same eigenvalues as $f$. Note that we have a surjective map $\bH_N\otimes \Q \rightarrow L_f$, where $L_f$ is the Hecke field of $f$, given by sending $T_v$ to the eigenvalue $a_v$ and $U_{\pri}$ to $\alpha_{\pri}$. This gives rise to a maximal ideal $\m_f \subset \bH_N\otimes L_f$. If $\cM$ is a finite-dimensional vector space, then $\cM_{f}$ is the localisation of $\cM$ at $\m_f$.

Throughout, the superscript $\eps$ will denote either a choice of sign $\pm$ (when $K=\Q$) or an empty condition (when $K = F$), reflecting the fact that $\Q$ has one real embedding whilst $F$ has none.

\subsection{Generalised eigenspaces of modular forms}

It is well-known that in the $p$-regular case, the generalised eigenspaces $S_\lambda(\Gamma_0(N))_{f}$ are 1-dimensional. This uses Strong Multiplicity One and the fact that $p$-regularity means the $U_{\pri}$ operators are diagonalisable at level $N$ for all $\pri|p$. In the irregular case, we no longer have this.

Recall $N = Mp$, and $f_{\mathrm{new}}$ is new at level $M$. Let $f_{\mathrm{new}}^N \in S_\lambda(\Gamma_0(N))$ be $f_{\mathrm{new}}(z)$ considered to have level $N$. Then $f_{\mathrm{new}}^N$ is an eigenform for $\bH_N^{\mathrm{tame}}$, that is, away from $p$. 

\begin{proposition}\label{prop:dims1} Suppose $f_{\mathrm{new}}$ is $p$-irregular. Then:
	\begin{enumerate}
		\item If $K=\Q$, then
		\[
		\mathrm{dim}_{\C} \ S_\lambda(\Gamma_0(N))_{f} = 2, \hspace{12pt}\text{and}\hspace{12pt}   	\mathrm{dim}_{\C} \ S_\lambda(\Gamma_0(N))[f] = 1.
		\]
		\item If $K = F$ is imaginary quadratic, then
		\[
		\mathrm{dim}_{\C} \ S_\lambda(\Gamma_0(N))_{f} = 4, \hspace{12pt}\text{and}\hspace{12pt}   	\mathrm{dim}_{\C} \ S_\lambda(\Gamma_0(N))[f] = 1.
		\]
	\end{enumerate}
In both cases the eigenspaces $S_\lambda(\Gamma_0(N))[f]$ are equal to $\C f$.
\end{proposition}

\begin{proof}
	First we work over $K = \Q$. By strong multiplicity one, we know that $S_\lambda(\Gamma_0(M))[f_{\mathrm{new}}]$ is a line, where we consider instead the action of the tame Hecke algebra $\bH_M^{\mathrm{tame}}$. Stabilisation at $p$ commutes with prime to $p$ Hecke operators, so the subspace of $S_\lambda(\Gamma_0(N))$ on which $\bH_N^{\mathrm{tame}}$ acts as $f$ is 2-dimensional, spanned by $f_{\mathrm{new}}^N$ and $f_{\mathrm{new}}(pz)$.  Moreover, since $U_p$ has a single eigenvalue $\alpha_p$ on this space and $\bH_N^{\mathrm{tame}}$ acts semi-simply on $S_\lambda(\Gamma_0(N))$, it follows that the generalised eigenspace $S_\lambda(\Gamma_0(N))_{f} $ for $\bH_N$ is equal to the eigenspace of the character of $\bH_N^{\mathrm{tame}} \to \C$ associated to $f_{\mathrm{new}}$, and is thus also 2-dimensional, spanned by $f_{\mathrm{new}}^N$ and $f_{\mathrm{new}}(pz)$. 
	
	To see that the usual eigenspace is a line, we show that $U_p$ does not act semisimply. By a standard calculation (see, for example, \cite[\S9.2]{RS17}), in the above basis the matrix of $U_p$ is $\smallmatrd{a_p}{1}{-p^{k+1}}{0}$, where $a_p = a_p(f_{\mathrm{new}})$ is the $T_p$-eigenvalue of $f_{\mathrm{new}}$. In particular, we have
	\[
	(U_p - \alpha_p) f_{\mathrm{new}}^N = (a_p - \alpha_p) f_{\mathrm{new}}^N - p^{k+1} f_{\mathrm{new}}(pz).
	\]
	Now, note that using irregularity, we have $(X - \alpha_p)^2 = X^2 - a_pX + p^{k+1}$, that is, $2 \alpha_p = a_p$ and $\alpha_p^2 = p^{k+1}$. Thus
	\[
	(U_p - \alpha_p) f_{\mathrm{new}}^N(z) = \alpha_p\big[f_{\mathrm{new}}^N(z) - \alpha_p f_{\mathrm{new}}(pz)\big] = \alpha_p f(z),
	\]
	by definition of the $p$-stabilisation $f$. In the basis $\{f,f_{\mathrm{new}}^N\}$, therefore, the matrix of $U_p$ is 
	\begin{equation}\label{eq:Up matrix}
	U_p = \matrd{\alpha_p}{\alpha_p}{0}{\alpha_p},
	\end{equation}
	so the $U_p$-eigenspace is 1-dimensional. Since $f$ is an eigenform it is thus generated by $f$.
	
	In automorphic terms, this says that if $\pi = \otimes_v' \pi_v$ is the automorphic representation of $\GL_2(\A_{\Q})$ generated by $f_{\mathrm{new}}$, then $\pi_p^{I_p}$ is two-dimensional with one-dimensional [$U_p = \alpha_p$]-eigenspace, where $I_p = \{\smallmatrd{a}{b}{c}{d} \in \GLt(\Zp) : p |c\}$ is the Iwahori subgroup.

	For (2), it is more convenient to use the language of automorphic representations. let $\pi_F = \otimes_v' \pi_{F,v}$ be the automorphic representation of $\GL_2(\A_F)$ corresponding to the base-change to $F$. Since $f_{\mathrm{new}}$ is base-change in this case, both $\pi_{F,\pri}$ and $\pi_{F,\pribar}$ are just copies of the representation $\pi_p$ of the corresponding classical newform, which has the shape above. Thus  $S_\lambda(\Gamma_0(N))_{f}$ is 4-dimensional, corresponding to the tensor product $\pi_{F,\pri}^{I_\pri} \otimes \pi_{F,\pribar}^{I_{\pribar}} = (\pi_p^{I_p})^{\otimes 2}$, and the eigenspace $S_{\lambda}(\Gamma_0(N))[f]$ is the 1-dimensional space spanned by $\phi_{\pri} \otimes \phi_{\pribar}$, where $\phi_{v}$ is a generator of the 1-dimensional $U_{v}$ eigenspace in $\pi_{v}^{I_v}$. Again, as $f$ is an eigenform it generates this line.
\end{proof}

\begin{corollary}\label{lem:Up - alpha}
	For $K = \Q$, we have
	\[
	(U_p - \alpha_p) f_{\mathrm{new}}^N = \alpha_p f, \ \ \ \ \ \ (U_p - \alpha_p)^2f_{\mathrm{new}}^N = 0.
	\]
	For $K = F$ imaginary quadratic, we have 
	\[
	(U_{\pri} - \alpha_p)(U_{\pribar} - \alpha_p) f_{\mathrm{new}}^N = \alpha_p^2 f, \ \ \ \ \ \ (U_{\pri} - \alpha_p)^2f_{\mathrm{new}}^N = (U_{\overline{\pri}} - \alpha_p)^2f_{\mathrm{new}}^N= 0.
	\]
	In both cases, the generalised eigenspaces $S_\lambda(\Gamma_0(N))_f$ are generated over $\C[\bH_N]$ by $f_{\mathrm{new}}^N$.
\end{corollary}
\begin{proof}
	In the case $K=\Q$, this is entirely proved above. For $K = F$ imaginary quadratic, note that $p$-stabilisation is a composition of $\pri$-stabilisation with $\pribar$-stabilisation. Let $f^{\pri}$ and $f^{\pribar}$ be the $\pri$- and $\pribar$-stabilisations of $f_{\mathrm{new}}$; then calculations exactly as above show that
	\begin{equation}\label{eq:bianchi 1}
	(U_{\pribar}-\alpha_p) f_{\mathrm{new}}^N = \alpha_p f^{\pribar}
	\end{equation}
	in the space of cusp forms. Similarly we have $(U_{\pri} - \alpha_p) f_{\mathrm{new}} = \alpha_p f^{\pri}.$
	But $\pri$ and $\pribar$ stabilisations commute, as do the $U_{\pri}$ and $U_{\pribar}$ operators; and $\pribar$-stabilisation commutes with $U_{\pri}$. Combining, we conclude that
	\begin{equation}\label{eq:bianchi 2}
	(U_{\pri} -\alpha_p) f^{\pribar} = \left[\pribar\text{-stabilisation of }(U_{\pri} - \alpha_p)f_{\mathrm{new}}^N\right] = \left[\pribar\text{-stabilisation of }\alpha_p f^{\pri}\right] = \alpha_p f,
	\end{equation}
	and the result follows by multiplying \eqref{eq:bianchi 2} by $\alpha_p$ and using \eqref{eq:bianchi 1}.
\end{proof}

\subsection{Classical cohomology}\label{sec:eichler-shimura}
Define a locally symmetric space 
\[
	\Y \defeq G(\Q) \backslash G(\A)/U_0(N)U_\infty Z_\infty^+,
\] 
where $U_0(N) \subset G(\widehat{\Z})$ is the open compact subgroup of matrices that are upper triangular mod $N$, $U_\infty$ is the standard maximal compact subgroup of $G(\R)$ and $Z_\infty^+$ is the centre of $G(\R)^+\defeq$ the identity  connected component of $G(\R)$.
Let $V_\lambda$ be the algebraic representation of $G$ of highest weight $\lambda$, and $V_\lambda^\vee$ its dual, which naturally gives rise to a local system $\sV_\lambda^\vee$ on $\Y$ via the action of $U_0(N)$ (see e.g.\ \cite[(I.1)]{BW20}). We have $V_\lambda(\C) = \mathrm{Sym}^k(\C^2)$ (resp. $\mathrm{Sym}^k(\C^2) \otimes [\mathrm{Sym}^k(\C^2)]^c$, for complex conjugation $c$) when $K$ is $\Q$ (resp.\ imaginary quadratic). The Hecke operators act on the cohomology via correspondences (see e.g.\ \cite[\S4]{Hid94}).

When $K=\Q$, we have an involution on $\uhp$ sending $z$ to $-\overline{z}$, which induces an involution on $\Y$, and hence an involution $\iota$ on the cohomology. For any local system $\sM$ on $\Y$ this in turn induces a decomposition
\[
	\hc{i}(\Y,\sM) = \hc{i}(\Y,\sM)^+ \oplus \hc{i}(\Y,\sM)^-
\]
on the cohomology, corresponding to the $\pm1$ eigenspaces for $\iota$, and this decomposition is Hecke-equivariant. Moreover for us $2$ will always be invertible on $\sM$, and then there are natural projectors $\mathrm{pr}^\pm = (1\pm\iota)/2$ to these subspaces. We use a superscript $\eps$ to denote a choice of sign $\pm$ (when $K=\Q$) or be an empty condition (when $K = F$ imaginary quadratic).

\subsubsection{The Eichler--Shimura isomorphism}
By composing \cite[Prop.~3.1]{Hid94} with the projector $\mathrm{pr}^\eps$ (\S8 \emph{op.\ cit.}), as $\bH_N$-modules we have the \emph{Eichler--Shimura isomorphism}
\begin{equation}\label{eq:ES}
\h_{\mathrm{cusp}}^1(\Y,\sV_\lambda^\vee(\C))^\eps \cong S_\lambda(\Gamma_0(N)),
\end{equation}
 where $\eps$ is either $\pm$ (for $K=\Q$) or nothing ($K$ imaginary quadratic). As a corollary, we see:

\begin{proposition}\label{prop:dims} Suppose $f_{\mathrm{new}}$ is $p$-irregular. Then:
	\begin{enumerate}
		\item If $K=\Q$, for each choice of sign $\pm$, 
		\[
			\mathrm{dim}_{\C} \ \hc{1}(\Y,\sV_\lambda^\vee(\C))^\pm_{f} = 2, \hspace{12pt}\text{and}\hspace{12pt} \mathrm{dim}_{\C} \  \hc{1}(\Y,\sV_\lambda^\vee(\C))^\pm[f] = 1.
		\]
		\item If $K = F$ is imaginary quadratic, then
				\[
		\mathrm{dim}_{\C} \ \hc{1}(\Y,\sV_\lambda^\vee(\C))_{f} = 4, \hspace{12pt}\text{and}\hspace{12pt} \mathrm{dim}_{\C} \  \hc{1}(\Y,\sV_\lambda^\vee(\C))[f] = 1.
		\]
	\end{enumerate}
\end{proposition}

\begin{proof}
	Using \cite[Lem.\ 3.15]{Clo90} and Strong Multiplicity One for $\GL_2$, there is a Hecke-equivariant isomorphism between $S_{\lambda}(\Gamma_0(N))_{f}$ (interpreted via the automorphic representation generated by $f$) and the cuspidal cohomology $\h^1_{\mathrm{cusp}}(Y_N,\sV_\lambda^\vee(\C))_{f}^\eps$. Moreover, the cuspidal cohomology injects into the compactly supported cohomology (e.g. \cite[p.\ 123]{Clo90}) with Eisenstein cokernel, so from \eqref{eq:ES} we deduce that it suffices to prove the analogous results for $S_{\lambda}(\Gamma_0(N))_{f}$.	The result then follows from Proposition \ref{prop:dims1}.
\end{proof}

\begin{definition}
For any $g \in S_\lambda(\Gamma_0(N))$, let $\phi_{g,\C}^\eps \in \hc{1}(Y_N,\sV_{\lambda}^\vee(\C))^\eps$ denote the attached (complex) cohomology class under the Eichler--Shimura isomorphism \eqref{eq:ES} and the inclusion $\h^1_{\mathrm{cusp}} \subset \h^1_{\mathrm{c}}$.
\end{definition}

\begin{corollary}
We have
\[
	\hc{1}(\Y,\sV_\lambda^\vee(\C))^\eps_{f} = \C[\bH_N] \cdot \phi_{f_{\mathrm{new}}^N,\C}^\eps\hspace{15pt}\text{and} \hspace{15pt} \hc{1}(\Y,\sV_\lambda^\vee(\C))^\eps[f] = \C \cdot \phi_{f,\C}^\eps.
\] 
\end{corollary}
\begin{proof}
	Immediate from Propositions \ref{prop:dims1} and \ref{prop:dims}, Lemma \ref{lem:Up - alpha} and Hecke-equivariance of \eqref{eq:ES}.
\end{proof}

\subsubsection{Periods and algebraic cohomology classes}
We now define cohomology classes attached to $f_{\mathrm{new}}$ and $f$ with algebraic coefficients. We define these periods using $f_{\mathrm{new}}$ at level $M$, and then $p$-stabilise to show that the entire generalised eigenspace at level $N$ can be made algebraic by scaling by the same periods.

Recall $N =Mp$, and let $U_0(M)$ and $Y_M$ be the direct analogues of $U_0(N)$ and $Y_N$ at level $M$. Using Eichler--Shimura at level $M$, the newform $f_{\mathrm{new}}$ determines a canonical class
\[
\phi_{f_{\mathrm{new}},\C}^\eps \in \hc{1}(Y_M,\sV_{\lambda}^\vee(\C))_f^\eps.
\]

\begin{proposition}
Let $E = \Q(f_{\mathrm{new}})$ be the Hecke field of $f_{\mathrm{new}}$. There exists a period $\Omega_f^\eps$ with
\[
	\phi_{f_{\mathrm{new}}}^\eps \defeq \frac{\phi_{f_{\mathrm{new}},\C}^\eps}{\Omega_f^\eps} \in \hc{1}(Y_M,\sV_\lambda^\vee(E))^\eps_f.
\]
\end{proposition}
\begin{proof}

Since $f_{\mathrm{new}}$ is a newform of level $M$, we have
\[
	\mathrm{dim}_E \hc{1}(Y_M,\sV_\lambda^\vee(E))^\varepsilon_{f_{\mathrm{new}}} = 1,
\]
defining an $E$-rational line inside the analogous line with $\C$-coefficients (see \cite[\S8]{Hid94}). Let $\phi_{f_{\mathrm{new}},E}^\varepsilon$ be a generator; this is of the form $(\Omega_f^\eps)^{-1} \cdot \phi_{f_{\mathrm{new}},\C}^\eps$ for some $\Omega_f^{\eps} \in \C^\times$. We then take $\phi_{f_{\mathrm{new}}}^\eps \defeq \phi_{f_{\mathrm{new}},E}^\eps$.
\end{proof}

We now transport this to level $N$. Recall $f_{\mathrm{new}}^N$ denotes the modular form in $S_\lambda(\Gamma_0(N))$ obtained simply by considering $f_{\mathrm{new}}(z)$ to have level $N$ rather than $M$, and that by Corollary \ref{lem:Up - alpha} this generates the generalised eigenspace at $f$ as a Hecke module. 
We also have a canonical class
\[
\phi_{f_{\mathrm{new}}^N,\C}^\eps \ \in \hc{1}(Y_N,\sV_\lambda^\vee(\C))^\eps_f
\]
given by Eichler--Shimura.
We have a natural quotient map $t: Y_N \to Y_M$; on cohomology, pullback under $t$ is equivariant with respect to all the Hecke operators away from $p$, and $t^*\phi_{f_{\mathrm{new}},\C}$ is precisely the image of $f_{\mathrm{new}}^N$ under the Eichler--Shimura isomorphism (at level $N$). With that in mind, we define the algebraic analogue:

\begin{definition}\label{def:phi alg}
Define
\[
\phi_{f^N_{\mathrm{new}}}^\varepsilon \defeq t^*\left(\frac{\phi_{f_{\mathrm{new}}^\eps,\C}}{\Omega_f^\varepsilon}\right) \in \hc{1}(Y_N,\sV_\lambda^\vee(E))^\varepsilon.
\]
\end{definition}
The following diagram and corollary summarise all of the above.

\[
\xymatrix@C=15mm@R=6mm{
f_{\mathrm{new}}\ar@{|->}@/_5pc/[ddd]  \ar@{|->}[r]\sar{d}{\in}&
\phi_{f_{\mathrm{new}},\C}^\eps \sar{d}{\in}&
\phi_{f_{\mathrm{new}}} \sar{d}{\in}\ar@{|->}@/^5pc/[ddd]  \ar@{|->}[l]   \\ 
S_\lambda(\Gamma_0(M)) \ar[d]^-{\mathrm{id}} \ar[r]^-{\mathrm{pr}^\eps \circ\text{E--S}} &
\hc{1}(Y_M,\sV_\lambda^\vee(\C))^\eps \ar[d]^-{t^*}  &
\hc{1}(Y_M,\sV_\lambda^\vee(E))^\eps_{f_{\mathrm{new}}} \ar[d]^-{t^*} \ar[l]^-{ \times \Omega_f^\eps}  \\
 S_\lambda(\Gamma_0(N)) \ar[r]^-{\mathrm{pr}^\eps \circ\text{E--S}} &
\hc{1}(Y_N,\sV_\lambda^\vee(\C))^\eps &
\hc{1}(Y_N,\sV_\lambda^\vee(E))^\eps_f  \ar[l]^-{ \times \Omega_f^\eps} \\
f_{\mathrm{new}}^N \sar{u}{\in} \ar@{|->}[r]  & 
\phi_{f^N_{\mathrm{new}},\C}^\eps \sar{u}{\in}  &
\phi_{f_{\mathrm{new}}^N}^\eps \sar{u}{\in} \ar@{|->}[l]
}
\]
\vspace{15pt}

\begin{corollary}\label{cor:gen eigenspace E}
The entire (level $N$) generalised eigenspace at $f$ is defined over $E(\alpha_p)$, and
\[
\hc{1}(\Y,\sV_\lambda^\vee(E(\alpha_p)))^\eps_{f} = E(\alpha_p)[\bH_N] \cdot \phi_{f_{\mathrm{new}}^N}^\eps.
\] 
The one-dimensional $\bH_N$-eigenspace in $\hc{1}(Y_N,\sV_\lambda^\vee(E(\alpha_p)))^\eps[f]$ is generated by
	\[
	\phi_f^\varepsilon \defeq \left\{\begin{array}{ll} \alpha_p^{-1}(U_p-\alpha_p)\phi_{f_{\mathrm{new}}^N}^\varepsilon &: K = \Q,\\
	\alpha_p^{-2}(U_{\pri} - \alpha_p)(U_{\pribar} - \alpha_p)\phi_{f_{\mathrm{new}}^N}^\varepsilon &: K\text{ imaginary quadratic}.\end{array}\right.
	\]
\end{corollary}
	
	\begin{proof}
		This is a formal consequence of the Hecke-equivariance of Eichler--Shimura, the rationality of Hecke operators, and Corollary \ref{lem:Up - alpha} (after adding $\alpha_p$ to the coefficient field).
	\end{proof}

After embedding $E(\alpha_p)$ into a sufficiently large finite extension $L/\Qp$, we henceforth always consider $\phi_f$ to have $p$-adic coefficients.

\subsubsection{Classical Hecke algebras}

 Eichler--Shimura descends to algebraic coefficients, i.e.\
\[
	S_\lambda(\Gamma_0(N),E(\alpha_p))_f \cong \hc{1}(Y_N,\sV_\lambda^\vee(E(\alpha_p)))^\eps_f.
\] 
This isomorphism is now non-canonical, depending on the choice of periods. We then pass to coefficients in $L$, obtaining a $p$-adic version 
\begin{equation}\label{eq:eichler-shimura L}
	S_\lambda(\Gamma_0(N),L)_f \cong \hc{1}(Y_N,\sV_\lambda^\vee(L))_f^\eps
\end{equation}
of Eichler--Shimura (see, for example, \cite[Prop.\ 3.18]{Bel12} for this $p$-adic version). From now on, we always work with coefficients in $L$, and thus suppress it from any further notation.

\begin{definition}\label{def:classical Hecke}
		Let $\bT_{\lambda,f}^\eps$ be the image of $\bH_N$ in $\mathrm{End}_{L}\hc{1}(\Y,\sV_\lambda^\vee)^\eps_{f}$.
\end{definition}

Hida duality, which remains true over $L$, is a perfect pairing between modular forms and the Hecke algebra sending $(g,T)$ to the leading Fourier coefficient of $Tg$. When composed with the Eichler--Shimura isomorphism we obtain a perfect pairing
\begin{equation}\label{eqn:perfect pairing}
	\hc{1}(\Y,\sV_\lambda^\vee)^\eps_{f} \times \bT_{\lambda,f}^\eps \longrightarrow L,
\end{equation}
and in particular $\bT_{\lambda,f}^\eps \cong [\hc{1}(\Y,\sV_\lambda^\vee)^{\eps}_{f}]^\vee$ (non-canonically, depending on the choice of periods). 

\begin{proposition}\label{prop:gorenstein}
	The Hecke algebra $\bT_{\lambda,f}^\eps$ is complete intersection (and hence Gorenstein).
\end{proposition}
\begin{proof}
	First note that all the tame Hecke operators act via scalars in $L$; thus by Corollary \ref{lem:Up - alpha} in the rational case we have $\bT_{\lambda,f}^\eps \cong L[X]/(X^2)$, where $X$ is the image of $U_p - \alpha_p$, and in the Bianchi case, it is the tensor product $L[X,Y]/(X^2,Y^2)$, where $X$ (resp.\ $Y$) is the image of $U_{\pri} - \alpha_p$ (resp.\ $U_{\pribar} - \alpha_p$). In both cases the rings are complete intersection.
\end{proof}

\begin{remark}\label{rem:quotient}
If $\cM \subset \hc{1}(\Y, \sV_\lambda)^\eps_{f} \subset \hc{1}(\Y, \sV_\lambda)^\eps$ is a non-trivial Hecke-stable subspace, let $I_{\cM} \defeq \mathrm{Ann}_{\bT_{\lambda,f}^\eps}(\cM)$, and $\bT(\cM) \defeq \bT_{\lambda,f}^\eps/I_{\cM}$ be the corresponding quotient of $\bT_{\lambda,f}^\eps$. Then \eqref{eqn:perfect pairing} restricts to a perfect pairing $\cM \times \bT(\cM) \to L$. By definition, $\bT(\cM)$ is isomorphic to the image of $\bH_N$ in $\mathrm{End}_L(\cM)$.

An important case we consider is that when $\cM$ is the unique one-dimensional Hecke-stable subspace, namely the eigenspace at $f$. Then $\bT(\cM)$ is the unique one-dimensional Hecke-stable quotient of $\bT_{\lambda,f}^\eps$ (that is, $I_{\cM} = \m_f$ is the maximal ideal at $f$).
\end{remark}

\section{The $p$-adic $L$-function of a single form}\label{sec:oc cohomology}
We recap the theory of overconvergent cohomology and its utility in attaching $p$-adic $L$-functions to modular forms. In the classical setting, this is a cohomological analogue of overconvergent modular forms. We are terse with the details; all of this material is explained in greater detail in \cite{PS11} (for the rational case) and \cite{Wil17} (the Bianchi case). For continuity in later sections, we maintain the notation of \cite{BW18} everywhere. Recall $\cO_p = \cO_K\otimes_{\Z}\Zp$.

\subsection{Overconvergent coefficients}
For $\Qp \subset L \subset \overline{\Q}_p$, let $\cA(\cO_p,L)$ be the module of locally analytic functions on $\cO_p$ with values in $L$. For each weight $\lambda$, this admits a natural left-action of the semigroup
\[
\Sigma_0(p) \defeq \left\{\smallmatrd{a}{b}{c}{d} \in M_2(\cO_p)\ :\ p|c, \ a \in \cO_p^\times, \ ad-bc \neq 0\right\}
\]
via
\[
\smallmatrd{a}{b}{c}{d} \cdot g(z) = \lambda(a+cz)g\left(\tfrac{b+dz}{a+cz}\right),
\]
where we consider $\lambda : \cO_p^\times \to \Qp^\times$ as a character in the usual way. When considering this space with this action, we denote it $\cA_\lambda$. We let $\cD(\cO_p,L)$ and $\cD_\lambda$ be the topological dual spaces (noting that $\cD_\lambda$ inherits a dual right-action). This induces an action of $U_0(N)$ via projection to the components at $p$, and we obtain a local system $\sD_\lambda$ on the cohomology (e.g.\ \cite[(I.1)]{BW20}). Dualising the natural inclusion $V_\lambda \subset \cA_\lambda$ as $\Sigma_0(p)$-modules gives rise to a Hecke-equivariant specialisation map
\[
\rho_\lambda : \hc{1}(\Y,\sD_\lambda) \longrightarrow \hc{1}(\Y,\sV_\lambda^\vee).
\]
This is equivariant for the involution $\iota$ at infinity.
\begin{theorem}\label{thm:control}
Suppose $v_p(\alpha_{\pri}) < k+1$ for every $\pri|p$. Then for every $i$, the restriction of $\rho_\lambda $ to the generalised eigenspaces at $f$ is an isomorphism $\rho_\lambda: \hc{i}(\Y,\sD_\lambda)_{f} \isorightarrow \hc{i}(\Y,\sV_\lambda^\vee)_{f}.$
\end{theorem}
In this generality, this is proved in \cite[Thm.~9.7]{BW_CJM}; though for $i = 1$, this was first proved in \cite{Ste94} (in the elliptic case) and \cite{Wil17} (in the Bianchi case).

We say forms $f$ satisfying this valuation condition have \emph{small} (or \emph{non-critical}) \emph{slope}.

\subsection{The $p$-adic $L$-function} \label{sec:p-adic L-function} 
The $p$-adic $L$-function of $f$ is naturally a locally analytic distribution on the narrow ray class group
\[
\cl_K^+(p^\infty) \defeq K^\times \backslash \A_K^\times/\widehat{\cO}_K^{(p),\times}K_\infty^+,
\]
where the superscript $(p)$ means away from $p$ and $K_\infty^+$ is the connected component of the identity in $(K\otimes\R)^\times$. By class field theory we identify these with distributions on the Galois group $\Gp$ of the maximal abelian extension of $K$ unramified outside $p\infty$, which is isomorphic to $\Zp^\times$ for $K=\Q$ and is a 2-dimensional $p$-adic analytic group for $K$ imaginary quadratic (corresponding to cyclotomic and anticyclotomic directions). For an $\cO_p$-algebra $R$, denote the $R$-valued locally analytic distributions on $\Gp$ by $\cD(\Gp,R)$.

The \emph{Mellin transform} is a map
\[
\mathrm{Mel}_\lambda : \hc{1}(\Y,\sD_\lambda(L)) \longrightarrow \cD(\Gp,L)
\]
described in \cite{PS11} and \cite{Wil17} (see also \cite[\S2.4]{BW18} for this language). It uses the identification of $\hc{1}$ with modular symbols and then formalises evaluation at $\{0\to\infty\}$.

Let $\phi_f^\eps \in \hc{1}(\Y,\sV_\lambda^\vee(L))^\eps[f]$ be the class constructed in Corollary \ref{cor:gen eigenspace E}. We assume $v_p(\alpha_{\pri}) < k+1$ for each $\pri|p$; if $f_{\mathrm{new}}$ is $p$-irregular, this is always satisfied since $v_p(\alpha_{\pri}) = (k+1)/2$. Then Theorem \ref{thm:control} means we can lift $\phi_f^\eps$ to a unique $\Phi_f^\eps \in \hc{1}(\Y,\sD_\lambda)^\eps[f]$. 

\begin{definition}
	If $K=\Q$, let $L_p^\pm(f) = \mathrm{Mel}_\lambda(\Phi_f^\pm)$, and let $\Phi_f \defeq \Phi_f^+ + \Phi_f^- \in \hc{1}(\Y,\sD_\lambda)$. Let 
	\[
	L_p(f) \defeq \mathrm{Mel}_\lambda(\Phi_f) \in \cD(\Gp,L).
	\]
\end{definition} 

The main results of \cite{PS11} (for $\Q$) and \cite{Wil17} (imaginary quadratic) were the following $p$-adic interpolation results for critical $L$-values; they appear as Proposition 6.5 and Theorem 7.4  respectively, where the precise interpolation factor is described. (See also equation \eqref{eqn:interpolation intro}).

\begin{theorem} $L_p(f)$ is the $p$-adic $L$-function of $f$; that is, for any Hecke character $\varphi$ of $K$ of $p$-power conductor and infinity type $0 \leq j \leq k$ (rational case) or $0 \leq (j_1,j_2) \leq (k,k)$ (Bianchi case), 
	\[
	L_p(f,\varphi) \defeq \int_{\Gp} \varphi_{p-\mathrm{fin}}(x) \cdot dL_p(f) = C(f,\varphi)\cdot \frac{\Lambda(f,\varphi)}{\Omega_f^\eps},
	\]
	where $C(f,\varphi)$ is an explicit factor, $\varphi_{p-\mathrm{fin}}$ is the $p$-adic avatar of $\varphi$, and $\eps$ depends on $\varphi$. Moreover, $L_p(f)$ satisfies a growth property making it unique with this interpolation property.
\end{theorem}

\begin{remark}
	If $K = \Q$, then $\varphi = \chi|\cdot|^j$ for $\chi$ a Dirichlet character of conductor $p^r$, and if $r > 1$
	\[
	C(f,\chi|\cdot|^j) = \frac{p^{r(j+1)}}{\alpha_p^r\tau(\chi^{-1})},
	\]
	where $\tau(\chi^{-1})$ is a Gauss sum. When $K = F$ is imaginary quadratic, the general formula is significantly more complicated, involving lots of extra notation that we will not need elsewhere and do not wish to define here: the full constant is given in \cite[Thm.\ 7.4]{Wil17}. In the special case where $\varphi$ is of the form $(\chi|\cdot|^j)\circ N_{F/\Q}$ for $\chi$ as above, which we require later, we have
	\[C(f,(\chi|\cdot|^j)\circ N_{F/\Q}) = 
	\frac{d^{j+1}p^{2r(j+1)}\#\cO_F^\times}{(-1)^{k}2\alpha_{p}^{2r}\tau((\chi\circ N_{F/\Q})^{-1})}
	\]
	(see \cite[Prop.\ 7.8]{BW18}), where $-d$ is the discriminant of $F/\Q$.
\end{remark}

Note that in the case $K=\Q$, each of $L_p^\pm(f)$ interpolate a different set of critical $L$-values. The critical values are at characters $\chi|\cdot|^j$, with $0 \leq j \leq k$, and $L_p^\pm(f)$ interpolates the values with $\chi(-1)(-1)^j = \pm 1$. In particular, $L_p^+(f)$ is supported on the $+$-half of weight space, that is, the union of the $(p-1)/2$ closed balls corresponding to even characters of $\Zp^\times$, and $L_p^-$ is supported on the $-$-half of odd characters. Each of $L_p^\pm(f)$ is well-defined only up to scaling the period $\Omega_f^\pm$ in $L^\times$, and since their support is disjoint, each can be scaled independently without affecting the other.

\begin{remark}\label{rem:alternative}
	We end this section by describing an alternative construction more closely related to variation in families. Since it is a map of $L$-vector spaces, the restriction of the Mellin transform
	\[
	\cM^\eps_\lambda[f] \defeq \hc{1}(\Y,\sD_\lambda(L))^\eps[f] \xrightarrow{\ \mathrm{Mel}_\lambda\ } \cD(\Gp,L)
	\]
	can be viewed as an element $\mathrm{Mel}_{\lambda}^\eps[f] \in \cD(\Gp,L)\otimes_L \cM_\lambda^\eps[f]^\vee$. We know $\cM_\lambda^\eps[f]$ and hence $\cM_\lambda^\eps[f]^\vee$ are one-dimensional $L$-vector spaces, and choosing a basis $\Xi_\lambda^\eps[f]$ of $\cM_{\lambda}^\eps[f]^\vee$ gives an isomorphism
	\[
		\cD(\Gp,L)\otimes_L \cM_\lambda^\eps[f]^\vee \isorightarrow \cD(\Gp,L)\otimes_L L \cong \cD(\Gp,L).
	\]
	By construction the image of $\mathrm{Mel}_\lambda^\eps[f]$ under this is exactly the distribution $L_p^\eps(f)$ above (up to the same indeterminancy, as we can scale our initial eigenclasses by elements in $L^\times$).

	Finally we give a conceptual reformulation of this that will be useful later. Let $\cM_{\lambda,f}^\eps \defeq \hc{1}(Y_N,\sD_\lambda)_f^\eps$ be the full generalised eigenspace. Analogously to Remark \ref{rem:quotient}, let $\cM \subset \cM_{\lambda,f}^\eps$ be a Hecke-stable submodule, and $\bT(\cM)$ the corresponding Hecke algebra. The restriction $\mathrm{Mel}_\lambda|_{\cM} : \cM \to \cD(\Gp,L)$ defines a canonical element $\mathrm{Mel}_{\cM} \in \cD(\Gp,L) \otimes_L \cM^\vee$. Suppose that:
	\[
		\cM^\vee\text{ is free of rank one over }\bT(\cM).
	\]
	Choosing a generator $\Xi_{\cM} \in \cM^\vee$ over $\bT(\cM)$ yields an isomorphism $\cD(\Gp,L)\otimes_L\cM^\vee \isorightarrow \cD(\Gp,L)\otimes_L\bT(\cM)$; let $\cL_{\cM} \in \cD(\Gp,L)\otimes_L\bT(\cM)$ be the image of $\mathrm{Mel}_{\cM}$,  well-defined up to scaling by $\bT(\cM)^\times$. 

Now suppose the eigenspace $\cM_{\lambda}^\eps[f] \subset \cM_{\lambda,f}
^\eps$ is a line; then every Hecke-stable submodule contains $\cM_{\lambda}^\eps[f]$, annihilated by the maximal ideal $\m_f \subset \bT(\cM)$, and Hida duality says $\cM_\lambda^\eps[f]$ is dual to the quotient $\bT(\cM)/\m_f \cong L$.  We conclude that reduction mod $\m_f$ at the level of Hecke algebras corresponds dually to restriction to $\cM_\lambda^\eps[f]$. In particular, we see that there exists $c_f^\eps \in L^\times$ such that
	\[
		\Xi_{\cM}\big|_{\cM_\lambda^\eps[f]} = c_f^\eps \cdot \Xi_\lambda^\eps[f].
	\]
	This $c_f^\eps$ measures the difference in choice of generator. Using the description of $L_p^\eps(f)$ in the first part of the remark, under the map
	\[
	\mathrm{sp}_f : \cD(\Gp,L)\otimes_L \bT(\cM) \to \cD(\Gp,L) \otimes_L \bT(\cM)/\m_f \cong \cD(\Gp,L)\otimes_L \cM_\lambda^\eps[f]^\vee \cong \cD(\Gp,L),
	\]
	we have $\mathrm{sp}_f(\cL_{\cM}) = c_f^\eps L_p^\eps(f)$.
\end{remark}

\section{The eigencurves at irregular points}\label{sec:geometry at irregular points}

\subsection{Overconvergent cohomology in families}\label{sec:dists in families}
The distributions $\cD_\lambda$, and the corresponding overconvergent cohomology groups, can be varied in $p$-adic families as $\lambda$ varies. In particular, let $\cW = \mathrm{Spf}(\Zp \lsem \Zp^\times \rsem)^{\mathrm{an}}$ be the weight space for $\GLt$. This embeds diagonally as the parallel weight subspace of the Bianchi weight space if $K$ is imaginary quadratic. Any affinoid subdomain $\Sigma = \mathrm{Sp}(\OS) \subset \cW $ then has an associated locally analytic tautological character $\chi_\Sigma : \cO_p^\times \to \OS^\times$. For a more detailed exposition, see \cite[\S3.2]{Bel12} or \cite[\S3.1]{BW18}.

For $\Sigma $ as above, we define $\cA_{\Sigma} \defeq \cA(\cO_p,L) \widehat{\otimes}_L \OS$, with $\Sigma_0(p)$-action given by 
\[
\smallmatrd{a}{b}{c}{d} \cdot g(z) = \chi_\Sigma(a+cz)g\left(\tfrac{b+dz}{a+cz}\right),
\]
and let $\cD_\Sigma \defeq \mathrm{Hom}_{\mathrm{cts}}(\cA_\Sigma,\OS)$. For $\lambda \in \Sigma$, corresponding to a maximal ideal $\m_\lambda$, we thus have
\begin{equation}\label{eq:spec dist}
\cD_\Sigma \otimes_{\OS} \OS/\m_\lambda \cong \cD_\lambda.
\end{equation}
 The induced dual (right) action of $\Sigma_0(p)$  on $\cD_\Sigma$ yields a local system $\sD_\Sigma$ on $Y_N$. The resulting cohomology groups are infinitely generated. To make computing in them more manageable, we use the following (see, for example, \cite[Lem.\ 3.4.14]{Urb11} or \cite[\S2.3, Prop.~3.1.5]{Han17}):
\begin{proposition}
	The matrix $\smallmatrd{1}{0}{0}{p}$ acts compactly on $\cD_\Sigma$. Hence, for any $h \geq 0$ and $\lambda \in \cW$, there exists a neighbourhood $\Sigma= \mathrm{Sp}(\Lambda)$ such that the groups $\hc{i}(\Y,\sD_\Sigma)$ admit a slope $\leq h$ decomposition with respect to $U_p$.
\end{proposition}

The small slope subspace $\hc{i}(\Y,\sD_\Sigma)^{\leq h}$ is finitely generated over $\OS$ \cite[Def.~2.3.1]{Han17}. 

 Recall from \S\ref{sec:notation} we have a maximal ideal $\m_f \subset \bH_{N} \otimes L$, the kernel of the map $\bH_N\to L$ sending each $T \in \bH_N$ to its eigenvalue on $f$. We get a maximal ideal of $\bH_N\otimes \Lambda$, also denoted $\m_f$, by pulling this back under $\mathrm{id} \otimes \newmod{\m_\lambda} : \bH_N\otimes \Lambda \twoheadrightarrow \bH_N \otimes L$. Let $\OSl$ be the localisation of $\OS$ at $\m_\lambda$, and $\hc{i}(\Y,\D_\Sigma)_f$ the localisation of $\hc{i}(\Y,\D_\Sigma)$ at $\m_f$.
For $h$ at least the slope of $f$ at $p$,  $\hc{i}(\Y,\D_\Sigma)_f$ is a finitely generated $\OSl$-submodule of $\hc{i}(\Y,\sD_\Sigma)^{\leq h}\otimes_{\OS}\OSl$. Thus we may freely use Nakayama's lemma when working with the slope $\leq h$ subspaces or after localisation at $f$.

\begin{notation}\label{not:simplify}
To simplify notation, let
\[
	\cM_{\lambda} \defeq \hc{1}(\Y,\sD_{\lambda}), \hspace{12pt} \cM_{\Sigma} \defeq \hc{1}(\Y,\sD_{\Sigma}).
\]
We use superscripts $\eps$, $\leq h$ to denote the $\eps$- and slope $\leq h$-parts respectively, write $\cM_{\lambda,f}$ (resp. $\cM_{\Sigma,f}$) for the localisations at $\m_f$, and $\cM_{\lambda}[f]$  for the actual $\bH_N$-eigenspace in $\cM_{\lambda,f}$.
\end{notation}

\subsection{Specialisation to single weights}
The specialisation $\cD_\Sigma \twoheadrightarrow \cD_\lambda$ of \eqref{eq:spec dist} induces a map $\mathrm{sp}_{\lambda}' : \cM_{\Sigma} \to \cM_{\lambda}$. To study this, define intermediate Hecke modules
\[
\overline{\cM}_{\Sigma}\defeq \cM_{\Sigma} \otimes_\Lambda \Lambda/\m_\lambda \hspace{12pt}\text{and}\hspace{12pt} \overline{\cM}_{\Sigma,f} \defeq \cM_{\Sigma,f} \otimes_{\Lambda_\lambda}\Lambda_\lambda/\m_\lambda.
\]
We will later use the following proposition, and its proof, to deduce that the classical or Bianchi eigencurve is Gorenstein at $f$, which will be crucial to our applications.

\begin{proposition}\label{prop:spec}
	Let $f$ be the $p$-stabilisation of a $p$-irregular newform of weight $\lambda$.
		Then $\mathrm{sp}_\lambda'$ induces a Hecke-equivariant injection $\mathrm{sp}_\lambda : \overline{\cM}_{\Sigma,f} \hookrightarrow \cM_{\lambda,f}$. Further:
		\begin{itemize}\setlength{\itemsep}{0pt}
			\item[(i)] If $K=\Q$, then $\mathrm{sp}_\lambda$ is an isomorphism.
			\item[(ii)] If $K$ is imaginary quadratic, then $\mathrm{dim}_L\overline{\cM}_{\Sigma,f}$ is either 2 or 4.
		\end{itemize}
\end{proposition}

\begin{proof} 
	Since $\mathbf{Z}_p \lsem \mathbf{Z}_p^{\times}\rsem$ is an UFD, its height one prime ideals are principal. In particular the maximal ideal $\m_\lambda = m_\lambda\OS$ is principal, and we have an associated short exact sequence \[0 \to \cD_\Sigma \xrightarrow{\ \times m_\lambda\ } \cD_\Sigma \to \cD_\lambda \to 0,\]
yielding a long exact sequence in cohomology
\begin{equation}\label{eq:long exact sequence}
\hc{i-1}(\Y,\sD_\lambda) \to \hc{i}(\Y,\sD_\Sigma) \xrightarrow{\ \times m_\lambda\ } \hc{i}(\Y,\sD_\Sigma) \to \hc{i}(\Y,\sD_\lambda) \to \hc{i+1}(\Y,\sD_\Sigma).
\end{equation}
 Localising this at $\m_f$ and truncating shows that specialisation induces injections
 \begin{equation}\label{eq:injective}
 	\hc{i}(\Y,\sD_\Sigma)_f \otimes_{\Lambda_\lambda}\Lambda_\lambda/\m_\lambda \hookrightarrow \hc{i}(\Y,\sD_\lambda)_f
 \end{equation}
 for all $i$. Injectivity of $\mathrm{sp}_\lambda$ is exactly this for $i = 1$.

Now suppose $K=\Q$; we turn to surjectivity in this case. Classical cusp forms do not appear in $\hc{2}(\Y,\sV_{\lambda}^\vee)$; so since $v_p(\alpha_p) = (k+1)/2 < k+1$, Theorem \ref{thm:control} implies $\hc{2}(\Y,\sD_\lambda)_{f} =0$. This can be seen alternatively as follows: we have $\hc{2}(\Y,\sD_\lambda) \simeq \mathrm{H}_{0}(\Y,\sD_\lambda)=\sD_\lambda/I^{\mathrm{aug}} \sD_\lambda$, where $I^{\mathrm{aug}}$ is the augmentation ideal of $\pi_1(\Y)$, and one can check as in \cite[Lem.~5.2,\S7]{PS12} that this space contains no cuspidal eigenpackets. It follows that $\hc{2}(\Y,\sD_\Sigma)_{f} = 0$ by Nakayama's lemma and \eqref{eq:injective} for $i=2$. Then the cokernel of $\mathrm{sp}_\lambda$ is 
\[	
	\hc{2}(\Y,\sD_\Sigma)_{f}[m_\lambda] = 0,
\]
proving (i). It remains to prove part (ii); we defer this to \S\ref{sec:hecke alg in families}.
\end{proof}

\begin{remarks}
This proof of (i) fails in the Bianchi setting, owing to the existence of classical eigenpackets in $\hc{2}$.  Further, we expect that $\mathrm{sp}_\lambda$ is \emph{not} surjective in this case; this would be implied by a conjecture of Calegari--Mazur \cite{CM09} (or more precisely, by a weaker form \cite[Conj.\ 5.13]{BW18}). Ultimately we do not determine whether $\mathrm{sp}_\lambda$ is surjective or not.

In the case $K$ imaginary quadratic, and $f_{\mathrm{new}}$ is $p$-regular and non-critical, $\mathrm{sp}_\lambda$ is shown to be an isomorphism in \cite[Thm.\ 4.5]{BW18}.
\end{remarks}

\subsection{The structure of Hecke algebras in families}\label{sec:hecke alg in families}

To complete the Bianchi case of Proposition \ref{prop:spec}, we translate the problem into the language of Hecke algebras. Let
\begin{equation}\label{eq:Hecke alg}
	\bT_{\Sigma} \defeq \bT\left(\cM_{\Sigma}\right), \ \ \ \overline{\bT}_{\Sigma} \defeq \bT(\overline{\cM}_{\Sigma}),
\end{equation} 
with decorations for localisations, $\eps$- and $\leq h$-parts as in Notation \ref{not:simplify} (e.g.\ $\bT_{\Sigma,f}^\eps$ is the localisation of $\bT(\cM_{\Sigma}^\eps)$ at $f$). The following relates these rings to Proposition \ref{prop:spec}(ii).

\begin{proposition}\label{prop:hecke isom}
	The natural surjection $\bT_{\Sigma,f}^\eps/\m_\lambda \isorightarrow \overline{\bT}_{\Sigma,f}^\eps$ is an isomorphism. In particular, if $r = \mathrm{dim}_L[ \overline{\cM}_{\Sigma,f}^\eps]$, then $\bT_{\Sigma,f}^\eps$ is free of rank $r$ over $\OSl$.
\end{proposition}
\begin{proof}
	We first treat the case $K=\Q$. Recall (Theorem \ref{thm:control}, \eqref{eq:eichler-shimura L}) that we have Hecke equivariant isomorphisms of $2$-dimensional $L$-vector spaces 
	\[
	\hc{i}(\Y,\sD_\lambda)_{f}^\eps \isorightarrow \hc{i}(\Y,\sV_\lambda^\vee)_{f}^\eps \isorightarrow  S_\lambda(\Gamma_0(N),L)_f.
	\]
	Let $S^\dagger_\lambda(\Gamma_0(N),L)$ denote the space of overconvergent cuspforms of weight $\lambda$; since $v_p(\alpha_p) < k+1$, by Coleman's classicality theorem the natural inclusion $S_\lambda(\Gamma_0(N),L)_f \hookrightarrow S^\dagger_\lambda(\Gamma_0(N),L)_f $ is an isomorphism.	 Since the cuspidal eigencurve constructed using overconvergent symbols is naturally isomorphic to the cuspidal Coleman--Mazur eigencurve (see \cite[Thm.\ 3.30]{Bel12}), both have the same local ring at $f$. Hence Hida's duality for Coleman families \cite[Prop.B.5.6]{Col97} yields that $\bT_{\Sigma,f}^\eps/\m_\lambda$ acts faithfully on  $S^\dagger_\lambda(\Gamma_0(N),L)_f  \simeq S_\lambda(\Gamma_0(N),L)_f$. Thus, $\bT_{\Sigma,f}^\eps/\m_\lambda$ is $2$-dimensional by Hida's duality for overconvergent cuspforms. Finally, the surjection $\bT_{\Sigma,f}^\eps/\m_\lambda \twoheadrightarrow \overline{\bT}_{\Sigma,f}^\eps$
	is necessarily an isomorphism since $\overline{\bT}_{\Sigma,f}^\eps$ is also $2$-dimensional.
	
	In the Bianchi case, there is no theory of overconvergent modular forms, and hence no Hida duality for overconvergent modular symbols. We prove results towards Hida duality in this case in the appendix, and use them to prove this isomorphism in Corollary \ref{cor:hecke isomorphism}.  
	
	The last statement follows from Nakayama's lemma, since $\Lambda_\lambda$ is a principal ideal domain.
\end{proof}

\subsubsection{Local pieces of eigencurves}\label{sec:local pieces}

By non-criticality and Remark \ref{rem:quotient}, the vector spaces $\overline{\cM}_{\Sigma,f}^\eps$ and $\overline{\bT}_{\Sigma,f}^\eps$ are dual, so have the same dimension. To prove Proposition \ref{prop:spec}(ii), by Proposition \ref{prop:hecke isom} it thus suffices to prove that in the Bianchi case, the ring $\bT_{\Sigma,f}$ has rank 2 or 4 over $\Lambda_\lambda$ (recalling $\eps = \varnothing$ in the Bianchi setting). We prove this by considering the geometry of the (classical and Bianchi) eigencurves.  It is now important to distinguish between the classical and Bianchi situations, so for the rest of \S\ref{sec:hecke alg in families}, we will distinguish the base-fields $\Q$ and $F$ in notation. We write $f$ for the ($p$-irregular) classical modular form and $\BCf$ for its base-change. Write $\cM_{\Q,\Sigma}$ (resp.\ $\cM_{F,\Sigma}$) for the classical (resp.\ Bianchi) overconvergent cohomology in families, with the usual decorations.

\begin{definition}
	Let $\cC$ be the cuspidal Coleman--Mazur eigencurve of tame level $\Gamma_0(M)$, equipped with the usual weight map $\kappa : \cC \to \cW$ to the weight space $\cW$ from \S\ref{sec:dists in families}. Let $\cE$ be the cuspidal parallel weight Bianchi eigenvariety (see \cite[\S5]{BW18}) of tame level $\Gamma_0(M\cO_F)$, with a map $\kappa_F : \cE \to \cW$ to the parallel weight line inside the natural two-dimensional Bianchi weight space $\cW_F$. (This is the union of the irreducible components in the full Bianchi eigenvariety that lie over the parallel weight line and contain a classical cuspidal point).
\end{definition}

Of particular importance for us is that local pieces of $\cC$ and $\cE$ can be constructed via overconvergent cohomology, as explained in \cite{AS08,Urb11,Han17}. For $K = \Q$ or $F$, let
\begin{equation}\label{eq:Hecke alg}
	\bT_{K,\Sigma} \defeq \bT(\cM_{K,\Sigma}), \ \ \ \overline{\bT}_{K,\Sigma} \defeq \bT(\overline{\cM}_{K,\Sigma}), 
\end{equation}
 In particular, we have
\[
\bT_{\Q,\Sigma}^{\leq h} \defeq \bT(\cM_{\Q, \Sigma}^{\leq h}) \hspace{12pt} \text{ and }\hspace{12pt} \bT_{F,\Sigma}^{\leq h} \defeq \bT(\cM_{F, \Sigma}^{\leq h}).
\]
Then
\[
\cC_{\Sigma}^{\leq h} \defeq \mathrm{Sp}(\bT_{\Q,\Sigma}^{ \leq h}) \subset \cC \hspace{12pt} \text{ and }\hspace{12pt} \cE_{\Sigma}^{\leq h} \defeq \mathrm{Sp}(\bT_{F,\Sigma}^{\leq h}) \subset \cE
\]
are local pieces, which for sufficiently large $h$ contain the points corresponding to $f$ and $\BCf$ respectively. Let $\cT_{\cC,f}$ (resp.\ $\cT_{\cE,f}$, resp.\ $\varLambda$) be the completed strictly Henselian local ring of $\cC$ (resp.\ $\cE$, resp.\ $\cW$) at $f$ (resp.\ $f_{/F}$, resp.\ $\lambda$). We see:
\begin{equation} \label{eq:completed local ring}
	\cT_{\cC,f}\text{ is the completion of the strict Henselisation of } \bT_{\Q,\Sigma,f} = \bT(\cM_{\Q,\Sigma}^{\leq h})_{\m_f},
\end{equation}
\begin{equation}\label{eq:completed local ring F}
	\cT_{\cE,f} \text{ is the completion of the strict Henselisation of } \bT_{F,\Sigma,f} = \bT(\cM_{F,\Sigma}^{\leq h})_{\m_{f/F}},
\end{equation}
\begin{equation}\label{eq:completed local ring W}
	\text{and both are modules over $\varLambda,$  the completion of the strict Henselisation of $\Lambda_\lambda$.}
\end{equation}
(Here we implicitly use that the algebraic and rigid localisations of the structure sheaf of a rigid space have the same completion \cite[\S7.3.2]{BGR}). Moreover in the rational case, since $f$ is cuspidal we have $\bT(\cM_{\Q,\Sigma,f}) = \bT(\cM_{\Q,\Sigma,f}^+) = \bT(\cM_{\Q,\Sigma,f}^-)$ (see \cite[Thm.\ 3.30]{Bel12}), so we can ignore signs here.

\subsubsection{The $p$-adic base-change map}\label{sec:surjective}
By \S\ref{sec:local pieces}, to complete Proposition \ref{prop:spec} we must show that the weight map $\kappa_F : \cE \to \cW$ has degree either 2 or 4 locally at $f_{/F}$. By Propositions \ref{prop:spec} and \ref{prop:hecke isom} in the classical case, we know that $\kappa : \cC\to \cW$ has degree 2 locally at $f$. To get the lower bound in the Bianchi case, we use \cite[Thm.\ 3.2.1, \S4.3]{JoNew}. This gives a natural map
\[
	\mathrm{BC} : \cC \longrightarrow \cE,
\]
induced from a map of abstract Hecke algebras $\mathrm{BC}^* : \bH_{N,F} \to \bH_{N,\Q}$, that interpolates the base-change transfer on classical points (so $\mathrm{BC}(f) = \BCf$). We write $\cE\bc$ for the image of this map. The map $\mathrm{BC}$ lies over $\cW$ in the sense that the following commutes:
\begin{equation}\label{base-change-funct}
\xymatrix{
\cC \ar[rr]^-{\mathrm{BC}}\ar[rd]_{\kappa} && \cE\bc \ar[ld]_{\kappa_F}\sar{r}{\subset} & \cE\ar[lld]^{\kappa_F}\\
& \cW  & & }
\end{equation}

Globally, the map $\mathrm{BC}^*$ will not be a surjection $\cO_{\cE} \to \mathrm{BC}_{*}(\cO_{\cC})$ (see Remark \ref{propertyBC} (3)). This encodes the fact that both $f$ and $f\otimes \chi_{F/\Q}$ have the same base-change (where $\chi_{F/\Q}$ is the quadratic character of $F$), so $\mathrm{BC}$ cannot globally be a closed immersion. Locally, however, we do have surjectivity. We prove this using deformation theory. Let
\[
	\mathrm{Ps}_{\cC}:G_\Q \longrightarrow \cO_{\cC}(\cC)
\]
 be the universal pseudo-character over $\cC$ (see, for example, \cite{JoNewExt}), which sends $\mathrm{Frob}_\ell$ to $T_\ell \in \cO(\cC)$ for any prime number $\ell \nmid Mp$, and let 
\[
\mathrm{Ps}_{\cE\bc}:G_F \longrightarrow \cO_{\cE\bc}(\cE\bc)
\]
 be the universal pseudo-character over $\cE\bc$, sending $\mathrm{Frob}_{\mathfrak{q}}$ to $T_{\mathfrak{q}} \in \cO(\cE^{\mathrm{bc}})$ for any prime $\mathfrak{q} \nmid Mp\cO_F$.

\begin{remark}\label{propertyBC}
\begin{enumerate}
	\item The structural map $\mathrm{BC}^{*}:\cO_{\cE\bc} \to \mathrm{BC}_{*}(\cO_{\cC})$ induces an equality
	\[
		 \mathrm{BC}^{*}(\mathrm{Ps}_{\cE\bc})=(\mathrm{Ps}_{\cC})_{\mid G_F}.
	\]
	
	\item For any $\mathfrak{p} \mid p$, $\mathrm{BC}^{*}(U_{\mathfrak{p}})=U_p \in \cO_{\cC}(\cC)^{\times}$ (recalling that $p$ splits in $F$). 

	\item Similarly, for any rational prime $q \nmid N$ split in $F$, we have $\mathrm{BC}^*(T_{\mathfrak{q}}) = T_q$ for any $\mathfrak{q}|q$. If $q$ is inert in $F$, then $\mathrm{BC}^*(T_\mathfrak{q})$ is a degree 2 polynomial in $T_q$ over $\cO_{\cW}(\cW)$ (see \cite[\S4.3]{JoNew}, noting the operator $[U_0(\n)\smallmatrd{q}{}{}{q}U_0(\n)]$ acts on $\hc{1}(\Y,\sD_\Sigma)$ by the scalar $\chi_\Sigma(q) \in \cO_{\cW}(\Sigma)$).
\end{enumerate}
\end{remark}

As in \eqref{eq:completed local ring}--\eqref{eq:completed local ring W}, let $\cT_{\cE\bc,f}$ be the completed (strictly Henselian) local ring of $\cE\bc$ at $\BCf$. Then \eqref{base-change-funct} induces a commutative diagram
\begin{equation}
\xymatrix{
	 \mathcal{T}_{\cC,f}	
	&&\cT_{\cE\bc,f}\ar[ll]_{\mathrm{BC}^*} & \cT_{\cE,f}\ar@{->>}[l] \\
	& \varLambda \ar[ul]^{\kappa^*} \ar[ur]^-{\kappa_F^*}\ar[urr]_{\kappa_F^*} 
	 &&
}
\end{equation}

\begin{prop}\label{prop:base-change surj}
	The base-change map $\cT_{\cE,f} \to \mathcal{T}_{\cC,f}$ is surjective. In particular, the base-change map $\cT_{\cE\bc,f} \to \mathcal{T}_{\cC,f}$ is an isomorphism, $\mathrm{BC} : \cC\to \cE$ is locally a closed immersion at $f$, and $\kappa_F: \cE\to \cW$ has degree $\geq 2$ locally at $f_{/F}$.
\end{prop}
\begin{proof} First, note that $f$ cannot have CM by $K$. If it did, $f$ would be ordinary since $p$ splits in $K$; but this contradicts the irregularity of $f$ at $p$.

 Since $\cT_{\cE\bc,f}$ is the quotient of $\cT_{\cE,f}$ cutting out the image of $\mathrm{BC}$, it suffices to prove that the base-change map $\cT_{\cE\bc,f} \to \cT_{\cC,f}$ is surjective.  Denote the pushforwards of $\Ps_{\cC}$ and $\Ps_{\cE\bc}$ under localisation (at $f$ and $\BCf$) by
	\[
		\Ps_{\cC,f} : G_{\Q} \longrightarrow \cT_{\cC,f} \ \ \ \text{ and }\ \ \ \Ps_{\cE\bc,f} : G_F \longrightarrow \cT_{\cE\bc,f}
	\]
	 respectively. Let $\cT_{\cC,f}^{+}$ be the image of $\cT_{\cE\bc,f}$ in $\cT_{\cC,f}$, and let 
	 \[
	 	\Ps_{\cC,f}^+ : G_F \longrightarrow \cT_{\cC,f}^+
	 \]
	 be the pushforward of $\Ps_{\cE\bc,f}$ under $\cT_{\cE\bc,f} \to \cT_{\cC,f}^+$. Since $\mathrm{BC}^{*}(\mathrm{Ps}_{\cE\bc})=(\mathrm{Ps}_{\cC})_{\mid G_F}$, we see $\Ps_{\cC,f}^+$ is exactly the restriction of $\Ps_{\cC,f}$ to $G_F$, and that $\cT_{\cC,f}^{+}$ is topologically generated over $\varLambda[U_p]$ by the image of $\mathrm{Ps}_{\cC,f}(G_F)$.
	
	As the restriction of a $G_{\Q}$-representation, $\mathrm{Ps}_{\cC,f}^+$ is invariant under the action by conjugation of $G_\Q$, and modulo the maximal ideal we have
	\[
		 \mathrm{Ps}_{\cC,f}^+ \equiv (\mathrm{tr} \rho_f)|_{ G_F} \newmod{\m_{\cC,f}^+},
	\]
	as $G_F$-representations, where $\rho_f:G_\Q \to \GL_2(\overline{\Q}_p)$ is the $p$-adic Galois representation attached to $f$ and $\m_{\cC,f}^+$ is the maximal ideal of $\cT_{\cC,f}^+$. Since $f$ does not have CM by $F$, the restriction $(\rho_f)_{\mid G_F}$ is absolutely irreducible. Then the main result of \cite{Nys96} yields that there exists a deformation 
	\[
		\rho_{\cC,f}^+:G_F \longrightarrow \GL_2(\cT_{\cC,f}^+)
	\]
	 of $(\rho_f)_{\mid G_F}$ such that $\mathrm{tr} \rho_{\cC,f}^+=\mathrm{Ps}_{\cC,f}^+$.  

Since we use strictly Henselian completed local rings (i.e.\ with residue field $\overline{\Q}_p$), we have
 \[
\rH^2(\Gal(F/\Q),(\cT_{\cC,f}^+)^\times)=(\cT_{\cC,f}^+)^{\times}/(\cT_{\cC,f}^+)^{\times,2} = 0,
\]
the final equality being Hensel's lemma. Thus by \cite[Thm.\ A.1.1]{Hid95} we can extend $\rho_{\cC,f}^+$ to a deformation 
\[
	\widetilde\rho_{\cT_{\cC,f}^+}:G_\Q \to \GL_2(\cT_{\cC,f}^+)
\]
 of $\rho_f$. (Conditions (AI$_{\mathrm{H}})$ and (Inv) \emph{op.\ cit}.\ are satisfied since $f_{/F}$ is cuspidal  base-change).  Any other $G_\Q$-extension of $\rho_{\cC,f}^+$ is a twist of $\widetilde\rho_{\cT_{\cC,f}^+}$ by the quadratic character $\chi_{F/\Q}$ of $\Gal(F/\Q)$. 
 
 On the other hand, using \cite{Nys96} again yields that $\mathrm{Ps}_{\cC,f}$ is the trace of a deformation $\rho_{\cC,f}:G_\Q \to \GL_2(\mathcal{T}_{\cC,f})$ of $\rho_f$, and that $\rho_{\cC,f}$ is an extension to $G_\Q$ of $\rho_{\cC,f}^+$. By the above remarks, $\rho_{\cC,f}$ is equal to $\widetilde\rho_{\cT_{\cC,f}^+}$ up to a twist by a quadratic character of $\Gal(F/\Q)$. 
 
 Finally, $\mathcal{T}_{\cC,f}$ is generated over $\varLambda$ by $U_p$ and the trace of $\rho_{\cC,f}$. Since $p$ is split, we know that $U_p \in \cT_{\cC,f}^+$ by Remark \ref{propertyBC}(2); and the submodule generated by the traces of $\rho_{\cC,f}$ and $\widetilde\rho_{\cT_{\cC,f}^+}$ are the same since they possibly differ only by quadratic twist. Hence $\mathcal{T}_{\cC,f} = \cT_{\cC,f}^+= \mathrm{Im}(\cT_{\cE^{\mathrm{bc}},f})$, as required.
 
 Given the surjectivity, $\mathrm{BC}$ is locally a closed immersion at $f$ by \cite[7.3.3, Prop.\ 4]{BGR}. Hence $\mathrm{deg}(\kappa_F) \geq \mathrm{deg}(\kappa) = 2$ locally at $f_{/F}$.
\end{proof}

\begin{corollary}\label{cor:base-change surj}	The base-change map $\bT_{F,\Sigma,f} \to \bT_{\Q,\Sigma,f}$ is surjective.
\end{corollary}
\begin{proof}
	By the above, there exists an affinoid neighbourhood $V_F$ of $f_{/F}$ in $\cE$ such that $\mathrm{BC}: \mathrm{BC}^{-1}(V_F) \to V_F$ is a closed immersion, that is, $\cO(V_F) \to \cO(\mathrm{BC}^{-1}(V_F))$ is surjective (noting $\mathrm{BC}^{-1}(V_F)$ is an affinoid since $\mathrm{BC}$ is finite). The result follows after localising at $f_{/F}$ and $f$ respectively.
\end{proof}

\subsubsection{The parity of the weight map at $f_{/F}$}\label{sec:parity}

Let $g_{\mathrm{new}}$ be a cuspidal Bianchi eigenform, generating a cuspidal automorphic representation $\pi_g$ of $\mathrm{GL}_2(\mathbb{A}_F)$. Let $c$ denote a lift of the non-trivial element of $\mathrm{Gal}(F/\Q)$ to an element of $G_{\Q}$. This acts on $\GL_2(\A_F)$ in the obvious way, and pre-composing we get a new cuspidal automorphic representation $\pi_g^c$. For any place $v$ of $F$, this will satisfy $\pi_{g,v}^c = \pi_{g,c v}$; that is, if $v$ lies above a prime of $\Q$ that splits in $F$, then $\pi_{g,v}^c = \pi_{g,\overline{v}}$. If $g_{\mathrm{new}}$ has level $\Gamma_0(M\cO_F)$ with $M \in \Z$, then the new vector $g_{\mathrm{new}}^c \in \pi_g^c$ is also a newform of level $\Gamma_0(M\cO_F)$; and then $c$ defines an involution on the space of Bianchi newforms of level $\Gamma_0(M\cO_F)$. We also have an analogous involution $c$ on $\bH_N$ which swaps $T_v$ and $T_{\overline{v}}$ for $(\ell) = v\overline{v}$ split and fixes $T_{\ell\cO_F}$ for $\ell$ inert.

\begin{lemma}\label{lem:base-change fixed}
	\begin{itemize}
		\item[(i)] The cuspidal automorphic representation is base-change if and only if $\pi_g = \pi_g^c$. 
		\item[(ii)] Let $g$ be a $p$-stabilisation of $g_{\mathrm{new}}$, corresponding to a system of eigenvalues $\phi_g : \bH_N\to \overline{\Q}_p$. Then $\phi_g$ appears in $\cE^{\bc}$ if and only if $\phi_g = \phi_g^c \defeq \phi_g \circ c$.
	\end{itemize}
\end{lemma}
\begin{proof}
	Part (i) is \cite[\S6.1]{Gel95}, proved in \cite{Lan80}. For part (ii), if $\phi_g$ appears in $\cE^{\bc}$ then $g$ is base-change and clearly $\phi_g = \phi_g^c$. Conversely, suppose $\phi_g = \phi_g^c$. Let $S$ denote the set of primes dividing $N$, and consider the Asai $L$-function $L^{\mathrm{As}}_S(g_{\mathrm{new}},s)$ with the Euler factors at primes in $S$ removed. By \cite{Asa77} (modified analogously to the Bianchi setting), this $L$-function has meromorphic continuation to $\C$, and has a pole if and only if $g_{\mathrm{new}}$ is base-change. This is proved by examining the Euler factors (as described in \cite[\S3]{Gha99}) and observing that when $g_{\mathrm{new}}$ is base-change, the Asai $L$-function factors as the product of a symmetric square $L$-function with the Riemann zeta function. Since $\phi_g(T_v) = \phi_g^c(T_v) = \phi_g(T_{cv})$ for all $v \notin S$, the same proof shows an analogous factorisation for $L^{\mathrm{As}}_S(g_{\mathrm{new}},s)$, which thus has a pole, and hence $g_{\mathrm{new}}$ is the base-change of some classical modular form $\tilde g_{\mathrm{new}}$. Finally, since $\phi_g(U_{\pri}) = \phi_g(U_{\pribar})$, we see $g$ is the base-change of a $p$-stabilisation of $\tilde g_{\mathrm{new}}$, and we are done.
\end{proof}

\begin{remark}
	An alternative argument is as follows: if $\phi_g = \phi_g^c$, then the associated Galois representation $\rho_g$ of $G_{F,S}$ attached to $g_{\mathrm{new}}$ is fixed under conjugation by $c$, and hence by \cite[Lem.\ A.1]{BW18} it admits an extension $\tilde\rho_g$ to $G_{\Q,S}$. If $g$ has sufficiently regular weight (as is always the case for $p$-irregular $g$) then $\rho_g$ -- hence $\tilde\rho_g$ -- is crystalline at $p$ by \cite[Lem.\ 3.15]{Che04}; hence $\tilde\rho_g$ is the representation of a classical newform $\tilde g_{\mathrm{new}}$ by Fontaine--Mazur, and $\tilde g_{\mathrm{new}}$ base-changes to $g_{\mathrm{new}}$.
\end{remark}

We get an induced involution on the eigencurve. The action at infinity induces an involution on the full Bianchi weight space $\cW_F$ that fixes the parallel weight subspace $\cW$. Since by definition every irreducible component of $\cE$ contains a classical cuspidal point, every such component necessarily contains a Zariski-dense set of cuspidal classical non-critical points, and by an application of $p$-adic Langlands functoriality (\cite[Thm.\ 3.2.1]{JoNew} or \cite[Thm.\ 5.1.6]{Han17}), we obtain an involution 
\begin{equation}\label{eq:involution eigencurve}
	\xymatrix{
			\cE\ar[rr]^-{\iota}\ar[rd]_-{\kappa_F} && \cE \ar[ld]^-{\kappa_F}\\
			& \cW &
		}
\end{equation}
of $\cE$ over $\cW$. We highlight that the involution acts trivially on $\cW$ since the central characters of the classical points of $\cE$ factor through the norm of $F/\Q$.

Let
\[
\mathrm{Ps}_{\cE}:G_F \longrightarrow \cO_{\cE}(\cE)
\]
be the $2$-dimensional universal pseudo-character carried by $\cE$. Then the involution $\iota:\cE \to \cE$ satisfies $\iota^*(\mathrm{Ps}_{\cE})=\mathrm{Ps}_{\cE}^c$, where $\mathrm{Ps}_{\cE}^c(g)=\mathrm{Ps}_{\cE}(c\cdot g \cdot c^{-1})$ for all $g \in G_F$. We let $\cE^{\nbc}$ be the closed analytic subset of $\cE$ given by the union of the irreducible components of $\cE$ which are not base-change; then as topological spaces, we have $\cE= \cE^{\mathrm{bc}} \cup \cE^{\nbc}$. We highlight that the property being a non-base-change point is open but not closed, and hence $\cE^{\nbc}$ might be seen as the Zariski analytic closure of the non-base-change points.

By Lemma \ref{lem:base-change fixed} the involution $\iota$ (preserves) and acts trivially on $\cE^{\mathrm{bc}}$, and (preserves) and acts non-trivially on  $\cE^{\nbc}$. It might be possible that $\BCf$ belongs also to $\cE^{\nbc}$, in which case $\BCf$ will be a crossing point between $\cE^{\nbc}$ and $\cE\bc$. 

We recall that $\varLambda$, $\cT_{\cE\bc,f}$ and $\cT_{\cE,f}$ are the completed (strictly Henselian) local rings of $\cW$ at $\kappa_F(\BCf)$, $\cE\bc$ at $\BCf$ and $\cE$ at $\BCf$ respectively. If $\BCf \in \cE^{\nbc} $, we let $\cT_{\cE^{\nbc},f}$ be the completed (strictly Henselian) local ring  of $\cE^{\nbc}$ at $\BCf$. Since the involution $\iota$ fixes $f_{\mid F}$, it acts on $\cT_{\cE,f}$, $\cT_{\cE^{\bc},f}$ and $\cT_{\cE^{\nbc},f}$; and from this and \eqref{eq:involution eigencurve}, the (finite flat) structure map $\kappa_F^*:\varLambda \to  \cT_{\cE,f}$ fits into diagrams
\begin{equation}\label{eq:involution hecke}
	\xymatrix{
		\cT_{\cE,f} && \cT_{\cE,f}\ar[ll]^-{\iota^*} \\
		& \varLambda \ar[ru]_-{\kappa_F^*}\ar[lu]^-{\kappa_F^*} &
	}\hspace{12pt}
	\xymatrix{
	\cT_{\cE^{\bc},f} && \cT_{\cE^{\bc},f}\ar[ll]^-{\iota^*} \\
	& \varLambda \ar[ru]_-{\kappa_F^*}\ar[lu]^-{\kappa_F^*} &
}
\hspace{12pt}
\xymatrix{
	\cT_{\cE^{\nbc},f} && \cT_{\cE^{\nbc},f}\ar[ll]^-{\iota^*} \\
	& \varLambda \ar[ru]_-{\kappa_F^*}\ar[lu]^-{\kappa_F^*} &
}.
\end{equation}

\begin{prop}\label{prop:2 or 4} One of the following assertions always holds true:
\begin{enumerate}
\item   $\mathrm{rk}_{\varLambda}\  \cT_{\cE,f} = 2$ and $\cT_{\cE,f} \simeq \cT_{\cE\bc,f} \cong \cT_{\cC,f}$, or
\item  $\mathrm{rk}_{\varLambda}\  \cT_{\cE,f} \geq 4$.
\end{enumerate}
\end{prop}
\begin{proof}

Recall that $\cT_{\cE\bc,f} \simeq \cT_{\cC,f}$ by Proposition \ref{prop:base-change surj}, and that the rank of $ \cT_{\cC,f}$ over $\varLambda$ is exactly two by Proposition \ref{prop:hecke isom}.  This implies that $\mathrm{rk}_{\varLambda}\ \cT_{\cE,f} \geq 2$ since we have a canonical projection $\ \cT_{\cE,f} \twoheadrightarrow \cT_{\cE\bc,f}$. Moreover $\cT_{\cE,f}$ is torsion-free, hence free, over the principal ideal domain $\varLambda$, so if $\mathrm{rk}_{\varLambda}\ \cT_{\cE,f} = 2$, then necessarily $\cT_{\cE,f} \simeq \cT_{\cE\bc,f} \simeq \cT_{\cC,f}$, and to prove the proposition it suffices to prove that the case $\mathrm{rk}_{\varLambda}\ \cT_{\cE,f} =3$ cannot happen.
 
Suppose $\mathrm{rk}_{\varLambda} \ \cT_{\cE,f} = 3$. Then $\cT_{\cE,f} \not\simeq \cT_{\cE^{\bc},f}$. Since the weight map $\varLambda \to  \cT_{\cE,f}$ is finite flat, $\Spec \cT_{\cE,f}$ is equidimensional of dimension $1$. Since $\cT_{\cE,f} \not\simeq \cT_{\cE\bc,f}$ (and both are reduced by \cite[Prop.\ 5.2]{BW18}), it follows that $\Spec \cT_{\cE\bc,f}$ is strictly contained in $\Spec \cT_{\cE,f}$. This implies that
$\Spec \cT_{\cE^{\nbc},f}$ is also equidimensional of dimension $1$. Since $\mathrm{rk}_{\varLambda} \cT_{\cE\bc,f} = 2$, we have $\mathrm{rk}_{\varLambda}\ \cT_{\cE^{\nbc},f} = 1$. In this case $\cT_{\cE^{\nbc},f} \simeq \varLambda$ and hence, since $\iota$ fixes $\varLambda$, by \eqref{eq:involution hecke} it fixes $\cT_{\cE^{\nbc}}$. It thus fixes a neighbourhood of $f_{/F}$ in $\cE^{\nbc}$. By assumption this neighbourhood is not base-change, so it must necessarily contain a non-base-change classical point, which is then fixed by $\iota$. But this contradicts Lemma \ref{lem:base-change fixed}. Hence $\mathrm{rk}_{\varLambda}\ \cT_{\cE^{\nbc}} \neq 1$ and $\mathrm{rk}_{\varLambda}\ \cT_{\cE,f} \neq 3$.
\end{proof}

We can finally complete the proof of Proposition \ref{prop:spec}(ii).
\begin{corollary}\label{cor:bianchi surjective} 
	In the Bianchi case, the map $\mathrm{sp}_\lambda : \overline{\cM}_{F,\Sigma,f} \hookrightarrow \cM_{F,\lambda,f}$ is either surjective or its image is $2$-dimensional.
\end{corollary}
\begin{proof}
	By Proposition \ref{prop:hecke isom}, we know that as a $\Lambda_\lambda$-module, $\bT_{F,\Sigma,f}$ is free of rank $r = \mathrm{dim}_L[\overline{\cM}_{F,\Sigma,f}]$, and moreover $r \leq 4$ (since $\mathrm{dim}_L[\cM_{F,\lambda,f}] = 4$). Passing to strictly Henselian completions, and working with cohomology over $\overline{\Q}_p$ rather than $L$, by directly analogous arguments we deduce that $\cT_{\cE,f}$ is also free of rank $r$ over $\varLambda$. But by Proposition \ref{prop:2 or 4} we have $r = 2$ or $4$ (and if $r=4$, $\mathrm{sp}_\lambda$ is surjective).
\end{proof}

\subsection{Gorensteinness of the eigencurve at irregular points} 
We now study the structure of overconvergent cohomology locally at $p$-irregular points of the eigencurve. Let $\Sigma = \Sp(\OS)\subset \cW$ be a \emph{nice} affinoid neighbourhood of $\lambda$, that is, with $\OS$ a principal ideal domain (see after \cite[Def.~3.5]{Bel12}).

\begin{proposition}
	The Hecke algebra $\overline{\bT}_{\Sigma,f}^\eps$ is Gorenstein. Hence $\bT_{\Sigma,f}^\eps$ is Gorenstein.
\end{proposition}
\begin{proof}
	If $K = \Q$, then applying $\pm$-projectors to Proposition \ref{prop:spec}(i) and Theorem \ref{thm:control}, we have Hecke-equivariant isomorphisms
\begin{equation}\label{eq:sign spec}
	 \mathrm{sp}_\lambda^1 : \cMbar_{\Sigma,f}^\eps \isorightarrow \cM_{\lambda,f}^\eps \isorightarrow \hc{1}(\Y,\sV_\lambda^\vee)_{f}^\eps.
\end{equation}
Thus the Hecke algebra $\overline{\bT}_{\Q,\Sigma,f}^\eps$ acting on $\cMbar_{\Q,\Sigma,f}^\eps$ is isomorphic to $\bT_{\Q,\lambda,f}^\eps$ (Definition \ref{def:classical Hecke}), so is Gorenstein by Proposition \ref{prop:gorenstein}.

If $K = F$ is imaginary quadratic, then by Proposition \ref{prop:2 or 4} and Corollary \ref{cor:bianchi surjective} either:
\begin{itemize}\setlength{\itemsep}{0pt}
	\item $\cT_{\cE,f} \cong \cT_{\cC,f}$, which is further isomorphic to $\cT_{\cC,f}^\eps$ by \cite[Thm.\ 3.30]{Bel12}. In this case we deduce $\bT_{F,\Sigma,f} \cong \bT_{\Q,\Sigma,f}^\eps$ via Corollary \ref{cor:base-change surj}, so $\overline{\bT}_{F,\Sigma,f} = \overline{\bT}_{\Q,\Sigma,f}$, which is Gorenstein by above;
	\item or specialisation is surjective, in which case $\overline{\bT}_{F,\Sigma,f} \cong \bT_{F,\lambda,f}^\eps$ which again is Gorenstein by Proposition \ref{prop:gorenstein}.
\end{itemize}  
Now $\bT_{\Sigma,f}^\eps$ is Gorenstein by \cite[47.21.6]{StacksProject}, as $\overline{\bT}_{\Sigma,f}^\eps = \bT_{\Sigma,f}^\eps/\m_\lambda =  \bT_{\Sigma,f}^\eps/(m_\lambda)$ by Proposition \ref{prop:hecke isom}.
\end{proof}

 As in Remark \ref{rem:quotient}, Hida duality on the classical spaces induces a perfect pairing
\begin{equation}\label{eq:perfect pairing}
	\cMbar_{\Sigma,f}^\eps \times \overline{\bT}_{\Sigma,f}^\eps \longrightarrow L,
\end{equation}
so we see that $\cMbar_{\Sigma,f}^\eps$ is the dualising module of $\overline{\bT}_{\Sigma,f}^\eps$. By the general formalism of Gorenstein rings, we deduce that $\cMbar_{\Sigma,f}^\eps$ and its dual $[\cMbar_{\Sigma,f}^{\eps}]^\vee$ are both free of rank one over $\overline{\bT}_{\Sigma,f}^\eps$.

Fix $h > 2v_p(\alpha_p)$, so that $f$ appears in the slope $\leq h$ overconvergent cohomology. Note that 
\begin{align}\label{eqn:spec to lambda}
	\cMbar_{\Sigma,f}^\eps = \hc{1}(\Y,\sD_\Sigma)^\eps_{f} \otimes_{\OSl}\OSl/\m_\lambda &\cong [\hc{1}(\Y,\sD_\Sigma)^\eps_{f} \otimes_{\bT_{\Sigma,f}^\eps}\bT_{\Sigma,f}^\eps]\otimes_{\OSl}\OSl/\m_\lambda\\ 
	&\cong \hc{1}(\Y,\sD_\Sigma)^\eps_{f} \otimes_{\bT_{\Sigma,f}^\eps}\bT_{\Sigma,f}^\eps/\m_\lambda = \cM_{\Sigma,f}^\eps\otimes_{\bT_{\Sigma,f}^\eps}\bT_{\Sigma,f}^\eps/\m_\lambda\notag,
\end{align}
which by the above is free of rank 1 over $\overline{\bT}_{\Sigma,f}^\eps.$  By Proposition \ref{prop:hecke isom}, it is thus free of rank 1 over $\bT_{\Sigma,f}^\eps/\m_\lambda$. By Nakayama's lemma applied to the (non-maximal) ideal $\m_\lambda$, we deduce that $\cM_{\Sigma,f}^\eps = \hc{1}(\Y,\sD_\Sigma)^\eps_f$ is generated by a single element over $\bT_{\Sigma,f}^\eps$.

\begin{proposition}\label{prop:free rank one}
$\hc{1}(\Y,\sD_{\Sigma})_f^\eps$ is free of rank 1 over $\bT_{\Sigma,f}^\eps$.
\end{proposition}

\begin{proof}
Since $\hc{1}(\Y,\sD_\Sigma)^\eps_f$ is projective over the principal ideal domain $\OSl$ (see \cite[Lem.\ 4.7]{BW18}), it is free (of some rank $r$). It follows that $\bT_{\Sigma,f}^\eps$ is also free as a submodule of $\mathrm{Mat}_r(\OSl)$. Moreover by \eqref{eq:perfect pairing} we know that $\bT_{\Sigma,f}^\eps/\m_\lambda$ and $\hc{1}(\Y,\sD_\Sigma)_f^\eps\otimes_{\OSl} \OSl/\m_\lambda$ have the same rank (namely, $r$) over $\OSl/\m_\lambda \cong L$, so Nakayama says that they are both free of rank $r$ over $\OSl$. 

From above, we know $\hc{1}(\Y,\sD_\Sigma)_f^\eps$ is cyclic over $\bT_{\Sigma,f}^\eps$, and thus has form $\bT_{\Sigma,f}^\eps/I$ for some ideal $I$. As no proper quotient of $\bT_{\Sigma,f}^\eps$ is free of rank $r$ over $\OSl$, we must have $I=0$, hence the result.
\end{proof}

Overconvergent cohomology defines a rigid coherent sheaf $\cF$ on $\cE$ \cite[Thm.\ 4.2.2]{Han17}, and Proposition \ref{prop:free rank one} describes its algebraic localisation (over the algebraic local ring $\bT_{\Sigma,f}^\eps$). We wish to lift this in families, for which we must use the rigid localisation. Let $\bT_{\Sigma,x_f}^\eps$ be the rigid localisation of $\cO(\cE_{\Sigma}^{\eps, \leq h})$ at $x_f$; this is a faithfully flat extension of $\bT_{\Sigma,f}^\eps$, and the two local rings have the same completion \cite[\S7.3.2(3)]{BGR}. Tensoring Proposition \ref{prop:free rank one} with $\bT_{\Sigma,x_f}^\eps$ over $\bT_{\Sigma,f}^\eps$ shows that $\cF$ is locally a line bundle at $x_f$.

\begin{corollary}\label{cor:lift to neighbourhood}
After possibly shrinking $\Sigma$, there exists a connected component $V^\eps = \Sp(\bT_{\Sigma,V}^\eps) \subset \cE_{\Sigma}^{\eps,\leq h}$ through $x_f$ such that:
	\begin{itemize}
		\item $\cM_{\Sigma,V}^\eps \defeq\hc{1}(\Y,\sD_\Sigma)^{\eps, \leq h} \otimes_{\bT^{\eps,\leq h}_{\Sigma}} \bT_{\Sigma,V}^\eps \ \ \text{is free of rank one over } \bT_{\Sigma,V}^\eps,$
		\item and $\bT_{\Sigma,V}^\eps$ is Gorenstein.
	\end{itemize}
\end{corollary}
\begin{proof}
	We obtain $V^\eps$ by rigid delocalisation, lifting a local isomorphism to a neighbourhood (e.g.\ \cite[Lem.~2.10]{BDJ17}). Since the non-Gorenstein locus is closed, and the eigencurve is Gorenstein at $x_f$, we can (after possibly further shrinking $\Sigma$) take $\bT_{\Sigma,V}^\eps$ Gorenstein.
\end{proof}

\section{Multi-variable $p$-adic $L$-functions}\label{sec:multi variable}
We now use Gorensteinness of $\bT_{\Sigma,V}^\eps = \bT(\cM_{\Sigma,V}^\eps)$ to prove our main result, the variation of $p$-adic $L$-functions over $V^\eps$. Combining Corollary \ref{cor:lift to neighbourhood} with the formalism of Gorenstein rings, we deduce that the $\OS$-linear dual $[\cM_{\Sigma,V} ^{\eps}]^\vee \defeq \mathrm{Hom}_{\OS}(\cM_{\Sigma,V}^\eps,\OS)$ is free of rank one over $\bT_{\Sigma,V}^{\eps}$. 
Also note that as $V^\eps$ is a connected component, $\cM_{\Sigma,V}^\eps$ is a direct summand of $\hc{1}(\Y,\sD_\Sigma)^{\eps,\leq h}$.

We can define the Mellin transform $\mathrm{Mel}_\Lambda$ in families, simply by considering coefficients in $\OS$ rather than $L$; see \cite[\S2.4]{BW18}. For any $\lambda \in \Sigma$, this fits into a commutative diagram 
\begin{equation}\label{eq:mellin spec}
\xymatrix@C=15mm{
	\hc{1}(Y_N,\sD_\Sigma) \ar[d]^{\mathrm{sp}_\lambda}	\ar[r]^-{\mathrm{Mel}_\Lambda} & \cD(\Gp, \Lambda)\ar[d]^{\mathrm{sp}_\lambda}\\
	\hc{1}(Y_N,\sD_\lambda) \ar[r]^-{\mathrm{Mel}_\lambda} & \cD(\Gp,L),
}	
\end{equation}
where the vertical maps are induced from reduction mod $\m_\lambda$. By restriction, we obtain a map
\begin{equation}\label{eq:restricted mellin}
	\mathrm{Mel}_\Lambda\big|_{\cM_{\Sigma,V}^\eps} : \cM_{\Sigma,V}^\eps \longrightarrow \cD(\Gp,\OS),
\end{equation}
which we consider naturally as a (canonical) element $\mathrm{Mel}_{\Sigma,V}^\eps \in \cD(\Gp,\OS) \otimes_{\OS} [\cM_{\Sigma,V}^{\eps}]^\vee.$ By picking a generator $\Xi_{\Sigma,V}^\eps$ of $[\cM_{\Sigma,V}^{\eps}]^\vee $ over $\bT_{\Sigma,V}^\eps$, we consider this as an element
\[
	L_p^\eps(V) \in \cD(\Gp,\OS)\otimes_{\OS} \bT_{\Sigma,V}^\eps \ \cong\ \cD(\Gp,\bT_{\Sigma,V}^\eps),
\]
which is now not canonical but well-defined up to $(\bT_{\Sigma,V}^\eps)^\times$ (corresponding to changing $\Xi_{\Sigma,V}^\eps$).

\begin{definition}
\begin{enumerate}
\item For $K = \Q$, define the \emph{two-variable $p$-adic $L$-function} to be 
\[
	L_p(V) \defeq L_p^+(V) + L_p^-(V) \in \cD(\Gp,\bT_{\Sigma,V}^\eps).
\]
\item For $K$ imaginary quadratic, the \emph{three-variable $p$-adic $L$-function} is $L_p(V) \in \cD(\Gp,\bT_{\Sigma,V}^\eps)$.
\end{enumerate}
\end{definition}

We now show that this interpolates the constructions at single points $y \in V^\eps(L)$. For such a point we have a natural specialisation map $\mathrm{sp}_y : \bT_{\Sigma,V}^{\eps} \to L$ given by reduction modulo $\m_y \subset \bT_{\Sigma,V}^\eps$. In particular, this induces a map
\[
	\mathrm{sp}_y : \cD(\Gp,\bT_{\Sigma,V}^\eps) \longrightarrow \cD(\Gp,L).
\]

\begin{theorem}
	Let $y = y_g \in V^\eps(L)$ be any classical point corresponding to a small slope cuspidal eigenform $g$ of weight $\lambda_g$. Then there exists a $p$-adic period $c_g^\eps$, depending only on $g$ and $\eps$, such that
\[
	\mathrm{sp}_y(L_p^\eps(V)) = c_g^\eps \cdot L_p^\eps(g).
\]
\end{theorem}

In particular, this applies when $y = x_f$ is the point corresponding to $f$, and in this case we can (and do) normalise so that $c_{x_f}^\eps = 1$ for each choice of $\eps$.

\begin{proof}
Let $y_g$ be such a classical point. This property is local at $g$, so we may restrict $L_p^\eps(V)$ to a neighbourhood of $g$, and without loss of generality assume $y_g$ is the only point of $V^\eps$ above $\lambda_g$. In this case reduction modulo $\m_{\lambda_g}$ is a map $\mathrm{sp}_{\lambda_g} : \cM_{\Sigma,V}^\eps \twoheadrightarrow \overline{\cM}_{\Sigma,g}^\eps \subset \cM_{\lambda,g}^\eps$.

Now reduction modulo $\m_y$ is the same as first reducing modulo $\m_{\lambda_g} \subset \m_y$ and then reducing modulo $\m_y$. We exploit this and the commutative diagram
\begin{equation}\label{eq:mellin spec 2}
	\xymatrix@C=15mm@R=3mm{
		\cM_{\Sigma,V}^\eps \ar[dd]^{\mathrm{sp}_{\lambda_g}}	\ar[r]^-{\mathrm{Mel}_\Lambda} &
		 \cD(\Gp, \Lambda)\ar[dd]^{\mathrm{sp}_{\lambda_g}}\\
		 &\\
		 \overline{\cM}_{\Sigma,g}^\eps \ar[r]^-{\mathrm{Mel}_{\lambda_g}} &
		 \cD(\Gp,L)\\
		 \cM_{\lambda,g}^\eps[g] \sar{u}{\subset} \ar[r]^-{\mathrm{Mel}_{\lambda_g}} &
		 \cD(\Gp,L) \sar{u}{=}
	}	
\end{equation}
 arising from \eqref{eq:mellin spec}. Note we already considered the bottom square in Remark \ref{rem:alternative}.
 
 As explained after \eqref{eq:restricted mellin}, the Mellin transform gives an element $\mathrm{Mel}_{\Sigma,V}^\eps \in \cD(\Gp,\Lambda) \otimes_\Lambda [\cM_{\Sigma,V}^\eps]^\vee.$ Similarly the restriction of $\mathrm{Mel}_{\lambda_g}$ to $\overline{\cM}_{\Sigma,g}^\eps$ defines a (canonical) element 
 \[
 	\overline{\mathrm{Mel}}_{\Sigma,g}^\eps \in \cD(\Gp,L) \otimes_L [\overline{\cM}_{\Sigma,g}^\eps]^\vee.
 \]
Since $\cM_{\Sigma,V}^\eps$ is finite free over $\OS$, we have a map $\mathrm{sp}_{\lambda_g} : [\cM_{\Sigma,V}^\eps]^\vee \to [\overline{\cM}_{\Sigma,g}^\eps]^\vee$ given by
\begin{align*}
	[\cM_{\Sigma,V}^{\eps}]^\vee \otimes_{\bT_{\Sigma,V}^\eps}\bT_{\Sigma,V}^\eps/\m_{{\lambda_g}}  \cong	[\cM_{\Sigma,V}^\eps]^\vee \otimes_{\OS}\OS/\m_{\lambda_g} \cong [\overline{\cM}_{\Sigma,g}^\eps]^\vee.
\end{align*}
By the commutativity of the top square in \eqref{eq:mellin spec 2}, we have
 \begin{equation}\label{eq:mellin commute}
 	\mathrm{sp}_{\lambda_g}(\mathrm{Mel}_{\Sigma,V}^\eps) = \overline{\mathrm{Mel}}_{\Sigma,g}^\eps \in \cD(\Gp,L)\otimes_L [\overline{\cM}_{\Sigma,g}^\eps]^\vee.
 \end{equation}
 Now we bring in the Hecke algebras. Since $[\cM_{\Sigma,V}^\eps]^\vee$ is free of rank one over $\bT_{\Sigma,V}^\eps$, after reducing modulo $\m_{\lambda_g}$ we see $[\overline{\cM}_{\Sigma,g}^\eps]^\vee$ is free of rank one over $\overline{\bT}_{\Sigma,g}^\eps$, and moreover the generator $\Xi_{\Sigma,V}^\eps$ we chose above reduces to a generator $\overline{\Xi}_{\Sigma,g}^\eps$. These choices of generators induce isomorphisms that sit in a commutative diagram
 \begin{equation}\label{eq:hecke commute}
 	\xymatrix{
 		\cD(\Gp,\Lambda)\otimes_\Lambda [\cM_{\Sigma,V}]^\vee \ar[r]^-{\simeq}\ar[d]^{\mathrm{sp}_{\lambda_g}} &
 		 \cD(\Gp,\Lambda) \otimes_\Lambda \bT_{\Sigma,V}^\eps \ar[d]^{\mathrm{sp}_{\lambda_g}}\\
 		 \cD(\Gp,L)\otimes_L[\overline{\cM}_{\Sigma,g}]^\vee \ar[r]^-{\simeq} &\cD(\Gp,L)\otimes_L \overline{\bT}_{\Sigma,g}^\eps.
 	}
\end{equation}
By definition $L_p^\eps(V)$ is the image of $\mathrm{Mel}_{\Sigma,V}^\eps$ along the top isomorphism. Define $\overline{\cL}_{\Sigma,g}^\eps$ to be the image of $\overline{\mathrm{Mel}}_{\Sigma,g}^\eps$ along the bottom isomorphism; combining \eqref{eq:mellin commute} and \eqref{eq:hecke commute}, we have $\mathrm{sp}_{\lambda_g}(L_p^\eps(V)) = \overline{\cL}_{\Sigma,g}^\eps.$
 
Now we use the bottom square. Note $\overline{\cL}_{\Sigma,g}^\eps$ is precisely the element $\cL_{\overline{\cM}_{\Sigma,g}^\eps}$ from Remark \ref{rem:alternative} (constructed using the generator $\overline{\Xi}_{\Sigma,g}^\eps$). As described in that remark, reducing mod $\m_g$ sends $\overline{\cL}_{\Sigma,g}^\eps$ to $c_g^\eps L_p^\eps(g)$ for some $c_g^\eps \in L^\times$, whence
\[
	\mathrm{sp}_y(L_p^\eps(V)) = \mathrm{sp}_{\lambda_g}(L_p^\eps(V)) \newmod{\m_g} = \overline{\cL}_{\Sigma,g}^\eps \newmod{\m_g} = c_g^\eps\cdot L_p^\eps(g),
\]
as required.
	\end{proof}

\begin{remarks}
	\begin{itemize}
			\item 
		Let $\sX(\Gp)$ be the rigid analytic space of $p$-adic characters on $\Gp$; it is 1-dimensional for $K=\Q$, and 2-dimensional for $K = F$ imaginary quadratic. Under the Amice transform (see \cite{ST01}), the distribution $L_p(V)$ gives a rigid analytic function $\cL_p : V\times \sX(\Gp)\to L$, which is the function described in the introduction.
		\item
From the interpolation propery satisfied by the $L_p^\eps(g)$, it follows immediately that $L_p(V)$ interpolates all the critical $L$-values of all classical forms $g$ in $V$, recalling that when $K=\Q$, the functions $L_p^+(g)$ and $L_p^-(g)$ will interpolate the critical $L$-values corresponding to even and odd characters respectively.
	\item
		Our construction is clearly inspired by \cite[Rem.~4.16]{Bel12}. It is possible to give a more explicit, and less conceptual, definition of $L_p^\pm$, by exhibiting an explicit eigenclass $\Phi_V^\eps \in \hc{1}(\Y,\sD_{\Sigma})^{\eps,\leq h}$ interpolating the eigenclasses $\Phi_g^\eps$ as $g$ varies in $V$. This is the approach explained in detail in \cite{Bel12}. One takes a generator $\Psi_V$ of $\cM_{\Sigma,V}^\eps$ over $\bT_{\Sigma,V}^\eps$ using Proposition \ref{prop:free rank one}, and observes that this interpolates generators of $\cMbar_{\Sigma,g}^\eps$ over $\overline{\bT}_{\Sigma,g}^\eps$ (via \eqref{eqn:spec to lambda}). The class $\Phi_V^\eps$ is then obtained by modifying by explicit Hecke operators at $p$, as in the definition of $\Phi$ in \cite[\S4.3.3]{Bel12}. The multi-variable $p$-adic $L$-function $L_p^\eps(V)$ is then the Mellin transform of $\Phi_V^\eps$.

	\end{itemize}
\end{remarks}

\section{$p$-adic Artin formalism}\label{sec:artin formalism}
We now distinguish between the classical and Bianchi cases; as such, let $F$ be an imaginary quadratic field in which $p$ splits, and fix $f$, as studied above, to be a classical modular form, with base-change $f_{/F}$. The classical complex $L$-functions of $f$ and $f_{/F}$ are related by Artin formalism
\[
	L(f_{/F},s) = L(f,s)L(f,\chi_{F/\Q},s),
\]
where $\chi_{F/\Q}$ is the quadratic character associated to $F/\Q$. For the $p$-adic picture, it is more convenient to adopt the formulation of Tate's thesis, in which case for a Hecke character $\varphi$ of $\Q$, it becomes
\begin{equation}\label{eqn:artin formalism}
	L(f_{/F},\varphi\circ N_{F/\Q}) = L(f,\varphi)L(f,\varphi\chi_{F/\Q}).
\end{equation}

We have a (two-variable) $p$-adic $L$-function $L_p(f_{/F})$ on $\Gp(F)$ attached to the base-change. Inspired by \eqref{eqn:artin formalism}, we define a (one-variable) locally analytic distribution $L_p^{\cyc}(f_{/F})$ on $\Gp(\Q)$ by
\[
\int_{\Gp(\Q)} \phi \cdot dL_p^{\cyc}(f_{/F}) \defeq \int_{\Gp(F)} (\phi\circ N_{F/\Q}) \cdot dL_p(f_{/F}).
\]
This operation is the restriction of $L_p(f_{/F})$ to the cyclotomic line inside $\cl_F(p^\infty)$.

Moreover, on the classical side, we have \emph{twisted} $p$-adic $L$-functions arising from twisted Mellin transforms; in particular, there is a $p$-adic $L$-function $L_p(f\otimes\chi_{F/\Q})$ interpolating the twisted critical $L$-values $L(f,\chi\cdot\chi_{F/\Q},j)$ for $\chi$ of $p$-power conductor. This is explained in detail in \cite[\S7.1]{BW18}.

We then have the following $p$-adic analogue of \eqref{eqn:artin formalism}.

\begin{theorem}\label{thm:artin formalism}
We have an equality
\[
	L_p^{\cyc}(f_{/F}) = L_p(f)L_p(f\otimes\chi_{F/\Q})
\]
of distributions on $\Gp(\Q)$.	
\end{theorem}
\begin{remark}
Note that both sides are really only well-defined up to scalars (up to changing the periods), so this equality is more properly seen as an equality of lines in the (uncountable dimensional) vector space $\cD(\Gal_p(\Q),L)$; this is explained in detail in \cite[\S7.2]{BW18}. 
\end{remark}
We actually deduce this from the analogous result in families. Again, using twisted Mellin transforms (from \cite[\S7.1]{BW18}), we obtain a twisted $p$-adic $L$-function $L_p(V\otimes \chi_{F/\Q})$ over $V$.

\begin{proposition}\label{prop:artin formalism families}
	Up to shrinking $\Sigma$ and renormalising by $\cO(V)^\times$, we have an equality
	\[
	L_p^{\cyc}(V_{/F}) = L_p(V)L_p(V\otimes\chi_{F/\Q})
	\]
	of $\Lambda$-valued distributions on $\Gp(\Q)$.	
\end{proposition}
\begin{proof}
	The slope of a Coleman family at $p$ is locally constant; shrinking $\Sigma$, we may assume:
	\begin{itemize}
		\item[(1)] the slope of every point of $V$ is $v_p(a_p(f))$,
		\item[(2)] and a Zariski-dense set of classical points $g$ in $V$ have `extremely' small slope, that is, $v_p(a_p(g)) = v_p(a_p(g)) < \frac{k_g+1}{2},$ where $g$ has weight $k_g + 2$.
	\end{itemize}
	It follows that at all such $g \in V$, and up to renormalising the $p$-adic periods, we have $L_p^{\cyc}(g_{/F}) = L_p(g)L_p(g\otimes\chi_{F/\Q})$, since in this case both sides interpolate the same classical $L$-values (using \eqref{eqn:artin formalism}) and are admissible of order $v_p(a_p(f))$ (which, under (2), is enough to determine them uniquely). This is proved fully in \cite[\S7.3]{BW18}. The result then follows exactly as in the proof of Prop.~7.9 \emph{op.\ cit}., where the result is given for $p$-regular forms.
\end{proof}

Theorem \ref{thm:artin formalism} then follows by specialising Proposition \ref{prop:artin formalism families} at the point $f$.


\begin{appendices}

	\titleformat{\section}
	{\large\bfseries\centering}
	{Appendix \thesection.}
	{0.5em}
	{}
	[]
	
\section{Towards Hida duality for modular symbols}\label{sec:pairing Hecke}

Let $\cM$ be a space with an action of the abstract Hecke algebra $\bH_N$ (see \S\ref{sec:notation}), and let $\bT(\cM)$ be the image of $\bH_N$ in $\mathrm{End}(\cM)$. In nice situations, one hopes for a duality between $\cM$ and $\bT(\cM)$. If $\cM$ is a space of (classical or Bianchi) modular forms of weight $\lambda$, for example, this is given by Hida duality, arising from the perfect pairing sending $(f,T)$ to the leading Fourier coefficient of $Tf$. This duality gives good control over the structure of the Hecke algebra.

We would like this duality in families. For $\GL_2/\Q$, there are two methods for variation in families, via overconvergent modular forms or symbols. The former has a good notion of $q$-expansions, allowing us to extend Hida duality in families \cite[Prop.B.5.6]{Col97}. By \footnote{\cite[Thm.\ 3.30]{Bel12} is an instance of \emph{$p$-adic Langlands functoriality}, see e.g.\ \cite[Thm.\ 3.2.1]{JoNew}. This says that if one has two eigenvarieties $\cE,\cE'$, a Zariski-dense set of classical points $\cE_{\mathrm{cl}} \subset \cE$, and a `Langlands functoriality' $\cE_{\mathrm{cl}} \to \cE'$, then this interpolates to a map $\cE \to \cE'$. This gives inverse maps between the Coleman--Mazur eigencurve of modular forms and the $\GL_2/\Q$-eigencurve of modular symbols, hence an identification of the respective Hecke algebras.} \cite[Thm.\ 3.30]{Bel12}, the Hecke algebras acting on modular forms and modular symbols coincide, so we (indirectly) get control on the Hecke algebra for modular symbols (as explained in Proposition \ref{prop:hecke isom}).

In the Bianchi setting, no theory of overconvergent modular forms exists; thus to study families we must use modular symbols, which have no notion of $q$-expansions. In this appendix, we prove results towards an analogue of Hida duality for modular symbols in families. 

Let $\Sigma = \mathrm{Sp}(\Lambda) \subset \cW$ be an affinoid in the weight space for $\GL_2$, which we identify with its image in the parallel Bianchi weight space. For further notation, we refer to \S\ref{sec:classical cohomology},\ref{sec:oc cohomology},\ref{sec:geometry at irregular points}. We use evaluation maps on modular symbols to exhibit a pairing between $\cM_\Sigma \defeq \hc{1}(Y_N,\sD_\Sigma)$ and its Hecke algebra $\bT_{\Sigma} = \bT(\cM_\Sigma)$, and prove that it is non-degenerate locally at any finite slope classical cuspidal eigenform. In the main text, this result provides control over the size of this Hecke algebra in Proposition \ref{prop:hecke isom}.

For simplicity of exposition, we focus on the (more difficult) Bianchi case. However, all of these results also hold, with directly analogous proofs, for modular symbols for $\GL_2/\Q$.

\subsection{Functionals on the cohomology}

We define \emph{evaluation maps}, functionals on the cohomology, as follows. Recall that $\Symb_\Gamma(\cD_\Sigma)$ is the space of $\Gamma$-invariant homomorphisms $\mathrm{Div}^0(\mathbf{P}^1(F)) \to \cD_\Sigma$; see \cite{Wil17,BW18} for further notation/definitions (for example of the arithmetic groups $\Gamma_i$).
\begin{enumerate}[(i)]
	\item We have an identification of the cohomology with modular symbols (see \cite[\S2.4]{BW18})
	\[
	\hc{1}(Y_N, \sD_\Sigma) \cong \bigoplus_{i \in \mathrm{Cl}(F)} \symb_{\Gamma_i}(\cD_\Sigma).
	\]
	\item Evaluation at $\{0\}-\{\infty\}$ defines a map $\oplus_{i \in \mathrm{Cl}(F)} \symb_{\Gamma_i}(\cD_\Sigma) \to \oplus_i \cD_\Sigma$.
	\item Taking the sum, we get a map $ \oplus_i \cD_\Sigma \to \cD_\Sigma$.
	\item Finally, taking total measure $\mu \mapsto \int\mu$ over $\cO_F\otimes_{\Z}\Zp$, we get a map $\cD_\Sigma \to \cO(\Sigma) = \Lambda$.
\end{enumerate} 
Write $\mathrm{Ev}_\Sigma : \hc{1}(Y_N,\sD_\Sigma) \to \Lambda$ for the composition of all of these maps. Since modular symbols are homomorphisms, $\mathrm{Ev}_\Sigma$ is a linear map of $\Lambda$-modules. Similarly, for any single weight $\mu \in \Sigma$ we have an analogous map $\mathrm{Ev}_\mu : \hc{1}(Y_N,\sD_\mu) \to \overline{\Q}_p$ of $\overline{\Q}_p$-vector spaces; and we have a commutative diagram
\begin{equation}\label{eq:compatibility evaluations}
	\xymatrix@C=15mm@R=7mm{
		\hc{1}(Y_N,\sD_\Sigma) \ar[r]^-{\mathrm{Ev}_{\Sigma}}\ar[d]^{\mathrm{sp}_\mu} & \cO(\Sigma)\ar[d]^{\mathrm{sp}_\mu} \\
		\hc{1}(Y_N,\sD_\mu) \ar[r]^-{\mathrm{Ev}_\mu} & \overline{\Q}_p,
	}
\end{equation}
where $\mathrm{sp}_\mu$ are the natural maps induced by evaluation at $\mu$.

\begin{remark}
	Alternatively, $\mathrm{Ev}_\mu$ is exactly the composition $\mathrm{Ev}_{\varphi,2} \circ \rho$ in \cite[\S11]{BW_CJM}, for $\varphi$ the trivial character. Both $\mathrm{Ev}_\mu$ and $\mathrm{Ev}_\Sigma$ are closely related to the total measure of the Mellin transform from \S\ref{sec:p-adic L-function} (though not identical; in $\mathrm{Mel}_\lambda$, there is a restriction to $(\cO_F\otimes_{\Z}\Zp)^\times$ that we omit above).
\end{remark}

Now let $g$ be any non-critical slope cuspidal Bianchi eigenform (over $F$) of parallel weight $\mu = (k_\mu, k_\mu)$, where $k_\mu \geq 2$. As in \cite[\S1.2]{Wil17}, via the Fourier expansion we have an attached $L$-function $L(g,s)$. (Note that this is not normalised; e.g.\ if $\gamma \in \C$ is a scalar, then $L(\gamma g,s) = \gamma L(g,s)$). Via \S\ref{sec:eichler-shimura}, after fixing a period $\Omega_g$ we obtain an attached algebraic class $\phi_g \in \hc{1}(Y_N,\sV_\mu^\vee(L))_g$. Via non-critical slope and Theorem \ref{thm:control}, we have a unique lift $\Phi_g$ of $\phi_g$ to $\hc{1}(Y_N,\sD_\mu)_g$.
\begin{proposition}\label{prop:ev non-zero}
	We have $\mathrm{Ev}_\mu(\Phi_g) \neq 0.$
\end{proposition}
\begin{proof}
	Write $\phi_g = (\phi_g^i) \in \oplus_{i \in \mathrm{Cl}(F)}\symb_{\Gamma_i}(V_\mu^\vee(L))$ in terms of modular symbols. From the definition, we see that $\mathrm{Ev}_\mu(\Phi_g) = \sum_{i}\phi_g^i\{0 \to \infty\}(1)$. (In the notation of \cite[\S2]{Wil17}, this is evaluation at $Y^k\overline{Y}^k$). By \cite[Thm.\ 2.11]{Wil17}, we then have
	\[
	\mathrm{Ev}_\mu(\Phi_g) = \sum_{i \in \mathrm{Cl}(F)}\phi_g^i\{0 \to \infty\}(1) = \left[\frac{(-1)^{k_\mu}8(\pi i)^2}{D \Omega_g\#\cO_F^\times}\right] \cdot L(g,1).
	\]
	This is at the bottom of the critical range. Since $k_\mu \geq 2$, this value is far from the centre of symmetry $s = 1+ k_\mu/2$ for the functional equation, and thus $L(g,1) \neq 0$ (since the functional equation reflects this value into a half-plane with an absolutely convergent Euler product).
\end{proof}

\subsection{Pairings between Hecke algebras and cohomology}

\begin{definition}
	By linearity of $\mathrm{Ev}_\Sigma$ and the Hecke action, we may define a pairing
	\[
	\langle-,-\rangle_\Sigma\ :\ \bT_{\Sigma}\ \times\ \hc{1}(Y_N,\sD_\Sigma) \longrightarrow \Lambda = \cO(\Sigma)
	\]
	of $\Lambda$-modules by $\langle T, \Phi\rangle_\Sigma \defeq \mathrm{Ev}_\Sigma(T\cdot \Phi).$ Similarly for each $\mu \in \Sigma$, letting $\bT_{\mu} \defeq \bT(\hc{1}(Y_N,\sD_\mu))$ we have a pairing
	\[
	\langle-,-\rangle_\mu\ : \ \bT_{\mu}\ \times\ \hc{1}(Y_N,\sD_\mu) \longrightarrow \overline{\Q}_p
	\]
	of $\overline{\Q}_p$-vector spaces given by $\langle T,\Phi\rangle_{\mu} \defeq \mathrm{Ev}_\mu(T\cdot \Phi)$.
\end{definition}

\begin{proposition}\label{prop:pairing spec}
	Let $T \in \bT_\Sigma$ and $\Phi \in \hc{1}(Y_N,\sD_\Sigma)$. If $\mu \in \Sigma$, then 
	\[
		\mathrm{sp}_\mu\big(\langle T ,\Phi\rangle_\Sigma\big) = \big\langle \mathrm{sp}_\mu(T), \mathrm{sp}_\mu(\Phi)\big\rangle_\mu.
	\]
\end{proposition}
\begin{proof}
	This follows from \eqref{eq:compatibility evaluations} and Hecke-equivariance of $\mathrm{sp}_\mu$ on cohomology. 
\end{proof}

Let $g \in S_\mu(\Gamma_0(N))$ be a non-critical slope cuspidal Bianchi eigenform, with $\mu = (k_\mu,k_\mu)$ for $k_\mu \geq 2$. Suppose that $g$ is the $p$-stabilisation of a $p$-regular newform for $\Gamma_0(M)$, with $N = Mp$. The restriction of $\langle-,-\rangle_\mu$ to $\hc{1}(Y_N,\sD_\mu)_g$ factors through the quotient $\bT_{\mu,g} = \bT(\hc{1}(Y_N,\sD_\mu)_g)$, giving a pairing
\[
\langle-,-\rangle_{\mu,g} \ : \ \bT_{\mu,g} \ \times \ \hc{1}(Y_N,\sD_\mu)_g \longrightarrow \overline{\Q}_p.
\]
\begin{proposition}
	For $g$ as above, the pairing $\langle-,-\rangle_{\mu,g}$ is perfect.
\end{proposition}
\begin{proof}
	The conditions ensure that $\hc{1}(Y_N,\sD_\mu)_g$ is a line, generated by $\Phi_g$. Hence $\bT_{\mu,g}$ is also a line, with each $T \in \bT_{\mu,g}$ acting as a scalar. By Proposition \ref{prop:ev non-zero} and linearity of $\mathrm{Ev}_\mu$, it follows that $\langle T, \Phi_g \rangle_\mu = 0$ only if $T= 0$ in the quotient $\bT_{\mu,g}$. Thus the restriction $\langle T, -\rangle_{\mu,g}$ defines an injection
	\[
	\bT_{\mu,g} \hookrightarrow [\hc{1}(Y_N,\sD_\mu)_g]^\vee.	
	\]
	Since both have dimension 1, this is an isomorphism and $\langle-,-\rangle_{\mu,g}$ is perfect.
\end{proof}

\subsection{Non-degeneracy}
Now let $f$ be any classical base-change cuspidal Bianchi eigenform of level $\Gamma_0(N)$ that is a $p$-stabilisation of a newform for $\Gamma_0(M)$, with $N = Mp$. For $h \gg 0$, via $p$-adic base-change there is a point $x_f$ in a local piece $\cE_{\Sigma}^{\leq h} = \mathrm{Sp}(\bT_{\Sigma}^{\leq h})$ of the parallel weight Bianchi eigenvariety (see \S\ref{sec:surjective} or \cite[\S5.2]{BW18}).  Let $V = \mathrm{Sp}(\bT_{\Sigma,V})$ be the connected component of $\cE_{\Sigma}^{\leq h}$ containing $x_f$. Shrinking $\Sigma$, we may assume $x_f$ is the only point of $V$ above $\lambda = \kappa_F(x_f)$. 
Since $f$ is base-change, the connected component $V$ is a rigid curve (possibly with multiple irreducible components). 

Recall the notion of a very Zariski-dense subset from, for example, \cite[Def.\ A.3]{BB}.

\begin{lemma}
	Up to shrinking $\Sigma$, there exist very Zariski-dense sets $\Sigma^{\mathrm{cl}} \subset \Sigma(\overline{\Qp})$ and $V^{\mathrm{cl}} \defeq \kappa_F^{-1}(\Sigma^{\mathrm{cl}}) \subset V(\overline{\Q}_p)$ such that each $y \in V^{\mathrm{cl}}$ is classical, corresponding to a cuspidal Bianchi eigenform $g$ such that:
	\begin{itemize}
		\item $g$ has non-critical slope at $p$;
		\item $g$ has weight $\mu = (k_\mu,k_\mu) \in \Sigma^{\mathrm{cl}}$ with $k_\mu \geq 2$;
		\item $g$ is the $p$-stabilisation of a $p$-regular newform of level $\Gamma_0(M)$.
	\end{itemize}
\end{lemma}
\begin{proof}
	Let $\Sigma^{\mathrm{cl}}$ be the set of classical weights $\mu = (k_\mu,k_\mu) \neq \lambda \in \Sigma$ with $k_\mu > 2h$ and $k_\mu \geq 2$. This set is very Zariski-dense in $\Sigma$. Let $V^{\mathrm{cl}} = \kappa_F^{-1}(\Sigma^{\mathrm{cl}})$. As $h < k_\mu/2 < k_\mu + 1$, every $y \in V^{\mathrm{cl}}$ has non-critical slope. No $y \in V^{\mathrm{cl}}$ of weight $\mu \in \Sigma^{\mathrm{cl}}$ can be irregular or new at any $\pri|p$, since irregular forms (resp.\ $\pri$-new forms) have slope $h = (k_\mu+1)/2$ (resp.\ $h = k_\mu/2$) at $\pri|p$ by (resp.\ \cite[Cor.\ 4.8]{BW17}); and these slopes are impossible, since $h < k_\mu/2$. Thus each such $y$ is a $p$-regular $p$-stabilisation.

It remains to show that each $y \in V^{\mathrm{cl}}$ is a $p$-regular $p$-stabilisation of a \emph{newform}. We argue analogously to \cite[Lem.\ 2.7]{Bel12}. We know every $y \in V^{\mathrm{cl}}$ arises from a newform at level $\n$ with $\n|M\cO_F$. For every such $\n$, applying $p$-adic Langlands functoriality to ''stabilisation to tame level $M$'' gives a closed immersion $\iota_{\n} : \cE_{\Sigma,\n}^{\leq h} \hookrightarrow \cE_{\Sigma}^{\leq h}$ from the tame level $\n$ to tame level $M\cO_F$ eigencurves, and since $x_f$ is new at level $M$, it is not in the image of any $\iota_{\n}$. We can thus shrink $\Sigma$ to avoid the image of $\iota_{\n}$ for any $\n|M\cO_F$, whence all points in a neighbourhood of $x$ are new at level $M$. 
\end{proof}

Let $\cM_{\Sigma,V} \defeq \hc{1}(Y_N,\sD_\Sigma)^{\leq h} \otimes_{\bT_{\Sigma}^{\leq h}} \bT_{\Sigma,V}$. As $V$ is a connected component, $\bT_{\Sigma,V}$ and $\cM_{\Sigma,V}$ are direct summands of $\bT_{\Sigma}$ and  $\hc{1}(Y_N,\sD_\Sigma)^{\leq h}$ respectively, and $\mathrm{Ev}_\Sigma$ can be restricted to $\cM_{\Sigma,V}$. We thus obtain a pairing
\[
\langle-,-\rangle_{\Sigma,V} \ :\  \bT_{\Sigma,V} \ \times \ \cM_{\Sigma,V} \longrightarrow \Lambda.
\]
If $y \in V(\overline{\Q}_p)$, we use a subscript $y$ for the (algebraic) localisation at the corresponding maximal ideal $\m_y \subset \bT_{\Sigma,V}$ (thus if $y$ is classical, corresponding to an eigenform $g$, we have $-_y = -_g$).

\begin{proposition}\label{prop:etale at good}
	Let $y \in V^{\mathrm{cl}}$ correspond to an eigenform $g$ of weight $\mu$. The weight map $\kappa_F : V \to \Sigma$ is \'etale at $y$, we have $\bT_{\Sigma,g}/\m_\mu \cong \bT_{\mu,g}$, and we have an isomorphism of 1-dimensional vector spaces
	\begin{equation}\label{eq:spec at g}
		\cM_{\Sigma,V}\otimes_{\bT_{\Sigma,V}} \bT_{\Sigma,V}/\m_y \cong \hc{1}(Y_N,\sD_\Sigma)^{\leq h}_g \otimes_{\Lambda_\mu} \Lambda_\mu/\m_\mu \cong \hc{1}(Y_N,\sD_\mu)_g.
	\end{equation}
\end{proposition}
\begin{proof}
	The \'etaleness is proved in \cite[Thm.\ 4.5]{BW18}, which also gives the second isomorphism in \eqref{eq:spec at g}. The first isomorphism is an immediate consequence of \'etaleness. The statement on the Hecke algebra is proved as in \cite[Prop.\ 4.6]{Bel12} (see also \cite[Thm.\ 6.13]{BW18}).
\end{proof}	

By Proposition \ref{prop:etale at good}, if $\mu \in \Sigma^{\mathrm{cl}}$, the fibre product $V \times_{\Sigma} \Sp(\Lambda/\m_\mu)$ is \'etale over $\Sp(\Lambda/\m_\mu)$, and hence
\[
\bT_{\Sigma,V} \otimes_{\Lambda}\Lambda/\m_\mu = \bT_{\Sigma,V}/\m_\mu \cong \bigoplus_{z \in \kappa_F^{-1}(\mu)} \bT_{\Sigma,V}/\m_z.
\]
Hence
\begin{align}\label{eq:direct sum y}
	\cM_{\Sigma,V} \otimes_{\Lambda}\Lambda/\m_\mu \cong \cM_{\Sigma,V}\otimes_{\bT_{\Sigma,V}}\bT_{\Sigma,V} \otimes_{\Lambda}\Lambda/\m_\mu
	&\cong \cM_{\Sigma,V}\otimes_{\bT_{\Sigma,V}} \bT_{\Sigma,V}/\m_\mu\notag\\
	& \cong 	 \cM_{\Sigma,V} \otimes_{\bT_{\Sigma,V}}\bigg[\bigoplus_{z \in \kappa_F^{-1}(\mu)}\bT_{\Sigma,V}/\m_z\bigg] \notag\\
	& 	\cong \bigoplus_{z \in \kappa_F^{-1}(\mu)} \hc{1}(Y_N,\sD_\Sigma) \otimes_{\bT_{\Sigma,V}}\bT_{\Sigma,V}/\m_z \notag\\
	&\cong \bigoplus_{z \in \kappa_F^{-1}(\mu)} \hc{1}(Y_N, \sD_\mu)_z,
\end{align}
where the last isomorphism is \eqref{eq:spec at g}. 

\begin{proposition}\label{prop:hecke injection families}
	We have an injective map of $\Lambda$-modules
	\begin{align*}
		\bT_{\Sigma,V} &\hookrightarrow [\cM_{\Sigma,V}]^\vee\\
		T &\mapsto \langle T, - \rangle_{\Sigma,V}.
	\end{align*}
\end{proposition}
\begin{proof}
	Suppose there exists $T \in \bT_{\Sigma,V}$ such that $\langle T, \Phi\rangle_{\Sigma,V} = 0$ for all $\Phi \in \cM_{\Sigma,V}$. Let $\mu \in \Sigma^{\mathrm{cl}}$. 
	For any fixed $y \in \kappa_F^{-1}(\mu)$, we may choose a generator $\Phi_y$ of the line $\hc{1}(Y_N,\sD_\mu)_y$, and consider this as an element of the direct sum $\oplus_{z \in \kappa_F^{-1}(\mu)} \hc{1}(Y_N,\sD_\mu)_z$ by taking 0 in the other summands. By \eqref{eq:direct sum y}, reduction modulo $\m_\mu$ defines a surjective map from $\cM_{\Sigma,V}$ to this direct sum, so we can lift to a class $\tilde\Phi_y \in \cM_{\Sigma,V}$ such that $\mathrm{sp}_\mu(\tilde\Phi_y)$ equals $\Phi_y$ in the $y$ summand of \eqref{eq:direct sum y}, and equals 0 in the summands for $z \neq y$. By Proposition \ref{prop:pairing spec}, we have
	\[
	0 = \mathrm{sp}_\mu\left(\left\langle T,\tilde\Phi_{y}\right\rangle_{\Sigma,V}\right) = \left\langle \mathrm{sp}_\mu(T), \mathrm{sp}_\mu(\tilde\Phi_y)\right\rangle_\mu = \langle \mathrm{sp}_y(T), \Phi_y\rangle_{\mu,y},
	\]
	where $\mathrm{sp}_y(T)$ is the image of $T$ in $\bT_{\Sigma,V}/\m_y$. Since $\langle -,-\rangle_{\mu,y}$ is perfect and $\Phi_y$ is a generator, this forces $\mathrm{sp}_y(T) = 0$, that is, $T \in \m_y \subset \bT_{\Sigma,V}$. Since $\mu$ and $y$ were arbitrary, we deduce that $T \in \bigcap_{y \in V^{\mathrm{cl}}} \m_y$, that is, $T$ is a rigid analytic function on $V$ that vanishes on $V^{\mathrm{cl}}$. Since this set is Zariski-dense, we conclude $T = 0$, and hence that the given map is injective, as desired.
\end{proof}

\subsection{Consequences for Hecke algebras} 
Let $f, V$ and $\Sigma = \mathrm{Sp}(\Lambda)$ be as above. Let $\cM_{\Sigma,f} \defeq \hc{1}(Y_N,\sD_\Sigma)_f$ and $\overline{\cM}_{\Sigma,f} \defeq \cM_{\Sigma,f} \otimes_{\Lambda_\lambda}\Lambda_\lambda/\m_\lambda$. Note that by truncating the long exact sequence attached to specialisation $\cD_\Sigma \to \cD_\lambda$ (see \eqref{eq:long exact sequence}), this injects into $\hc{1}(Y_N,\sD_\lambda)_f$ and thus has finite dimension.

\begin{proposition}\label{prop:hecke free rank s}
	Let $r = \mathrm{dim}\ \overline{\cM}_{\Sigma,f}$. Up to further shrinking $\Sigma$, the Hecke algebra $\bT_{\Sigma,V}$ is free of some rank $s \leq r$ over $\Lambda$, and the local Hecke algebra $\bT_{\Sigma,f}$ is free of rank $s$ over $\Lambda_\lambda$.
\end{proposition}
\begin{proof}
	By Nakayama's lemma, $\cM_{\Sigma,f}$ is generated by $r$ elements over $\Lambda_\lambda$. It is projective over $\OSl$ (see \cite[Lem.~4.7]{BW18}), hence free of rank $r$ as $\OSl$ is a principal ideal domain. Delocalising, up to shrinking $\Sigma$ and $V$ we have $\cM_{\Sigma,V}$ (hence $[\cM_{\Sigma,V}]^\vee$) is free of rank $r$ over $\Lambda$.	
	Further shrinking $\Sigma$ if necessary, we may always take $\Lambda$ to be a principal ideal domain (see discussion after \cite[Def.\ 3.5]{Bel12}). By Proposition \ref{prop:hecke injection families}, we know there is an injection $\bT_{\Sigma,V} \hookrightarrow [\cM_{\Sigma,V}]^\vee$ of $\Lambda$-modules; hence, as a submodule of a finite free module over a principal ideal domain, $\bT_{\Sigma,V}$ is also free over $\Lambda$, of some rank $s \leq r$. The local result follows by localising at $f$, since $x_f$ is the unique point of $V$ above $\lambda$.
\end{proof}
We obtain the following refinement of \cite[Lem.\ 6.3.4]{Che04}.
\begin{corollary}\label{cor:hecke isomorphism}
	If $f$ is non-critical, we have $\bT_{\Sigma,f}/\m_\lambda \cong \overline{\bT}_{\Sigma,f} \defeq \bT(\overline{\cM}_{\Sigma,f})$.
\end{corollary}
\begin{proof}
	The map $\hc{1}(Y_N,\sD_{\Sigma}) \twoheadrightarrow \hc{1}(Y_N,\sD_{\Sigma})\otimes_\Lambda \Lambda/\m_\lambda$ is Hecke equivariant, so we get a natural map $\bT_{\Sigma,f}/\m_\lambda \to \overline{\bT}_{\Sigma,f}$; it surjects since (by definition) $\overline{\bT}_{\Sigma,f}$ is the image of $\bH_N$ in the endomorphism ring.

	We know $\overline{\cM}_{\Sigma,f} \hookrightarrow \hc{1}(Y_N,\sD_\lambda)_f \cong \hc{1}(Y_N,\sV_\lambda^\vee)_f$, where the first injection follows from truncating the long exact sequence attached to $\cD_\Sigma \to \cD_\lambda$ (see Proposition \ref{prop:spec}) and the second isomorphism is non-criticality. By Hida duality and Remark \ref{rem:quotient}, we thus know $\overline{\bT}_{\Sigma,f}$ has dimension $r = \mathrm{dim}\ \overline{\cM}_{\Sigma,f}$ over $\Lambda/\m_{\lambda}$. Now $\bT_{\Sigma,f}/\m_\lambda$ has dimension $s \leq r$ by Proposition \ref{prop:hecke free rank s}, and surjects onto $\overline{\bT}_{\Sigma,f}$; thus $s = r$ and the map is an isomorphism.
\end{proof}

\end{appendices}

\bibliographystyle{alpha}
\bibliography{master_references}

\end{document}